\documentclass{article}

\PassOptionsToPackage{numbers, compress}{natbib}

\usepackage[preprint]{neurips_2021}

\RequirePackage{amsthm,amsmath,bm,graphicx,enumitem, relsize}
\usepackage[utf8]{inputenc} 
\usepackage[T1]{fontenc}    
\usepackage{hyperref}       
\usepackage{url}            
\usepackage{booktabs}       
\usepackage{amsfonts}       
\usepackage{nicefrac}       
\usepackage{microtype}      
\usepackage{xcolor}         
\RequirePackage[OT1]{fontenc}
\usepackage{amssymb,tabularx,multicol,multirow,booktabs,csquotes}
\usepackage{comment}

\usepackage{algorithm}
\usepackage{algpseudocode}

\newcommand\Algphase[1]{%
\vspace*{-.7\baselineskip}\Statex\hspace*{\dimexpr-\algorithmicindent-2pt\relax}\rule{\textwidth}{0.4pt}%
\Statex\hspace*{-\algorithmicindent}\textbf{#1}%
\vspace*{-.7\baselineskip}\Statex\hspace*{\dimexpr-\algorithmicindent-2pt\relax}\rule{\textwidth}{0.4pt}%
}

\newtheorem{theorem}{Theorem}
\newtheorem{lemma}{Lemma}
\newtheorem{remark}[theorem]{Remark}

\def \bbeta{\boldsymbol{\beta}}
\def \bx{\mathbf{x}}
\def \bX{\mathbf{X}}
\def \bbX {\mathbb{X}}

\def \bY{\mathbf{Y}}
\def \bbY {\mathbb{Y}}
\def \bZ {\mathbb{Z}}

\def \E{\mathbb{E}}
\def \P{\mathbb{P}}

\title{On Ensembling vs Merging: Least Squares and Random Forests under Covariate Shift}

\author{%
  Maya Ramchandran  \\
  Department of Biostatistics\\
  Harvard University\\
  Cambridge, MA 02138 \\
  \texttt{maya\_ramchandran@g.harvard.edu} \\
  \And
  Rajarshi Mukherjee \\
  Department of Biostatistics\\
  Harvard University\\
  Cambridge, MA 02138 \\
  \texttt{rmukherj@hsph.harvard.edu} \\
}

\begin{document}

\maketitle

\begin{abstract}
  It has been postulated and observed in practice that for prediction problems in which covariate data can be naturally partitioned into clusters, ensembling algorithms based on suitably aggregating models trained on individual clusters often perform substantially better than methods that ignore the clustering structure in the data. In this paper, we provide theoretical support to these empirical observations by asymptotically analyzing linear least squares and random forest regressions under a linear model. Our main results demonstrate that the benefit of ensembling compared to training a single model on the entire data, often termed 'merging', might depend on the underlying bias and variance interplay of the individual predictors to be aggregated. In particular, under both  fixed and high dimensional linear models, we show that merging is asymptotically superior to optimal ensembling techniques for linear least squares regression due to the unbiased nature of least squares prediction. In contrast, for random forest regression under fixed dimensional linear models, our bounds imply a strict benefit of ensembling over merging. Finally, we also present numerical experiments to verify the validity of our asymptotic results across different situations.
\end{abstract}

\section{Introduction}

\par Modern statistical applications often encounter situations where it is critical to account for inherent heterogeneity and substructures in data \cite{luo2010, verbeke1996, Dietterich2000, Bouwmeester2013}. Under the presence of heterogeneity in the distribution of covariates, it is natural to seek prediction strategies which are robust to such variability in order to accurately generalize to new data. In this regard, previous research has hinted at the possibility of improving prediction accuracy in such setups by pre-processing data through suitable clustering algorithms \cite{Ramchandran2020, Trivedi2015, deodhar2007}. Subsequently  applying ensembling frameworks to pre-clustered data has also been shown to produce critical advantages \cite{Patil2018, ramchandran2021}. In this paradigm, data are first separated into their component clusters, a learning algorithm is then trained on each cluster, and finally the predictions made by each single-cluster learner are combined using weighting approaches that reward cross-cluster prediction ability within the training set.  Learning algorithms that have been explored within this framework include Neural Networks, Random Forest, and (regularized) least squares linear regression  \cite{Patil2018, ramchandran2021, Sharkey1996, randomForest}. These research activities have imparted pivotal insights into identifying prediction algorithms which might benefit most from such ensembling techniques compared to ``merging" methods (referring to the predictions made by the chosen learning algorithm on the entire data) that ignore any underlying clusters in the covariate distribution. 
In this paper, we provide a crucial theoretical lens to further our understanding of ensembling methods on clustered data under a high dimensional linear regression setup with covariate shift defining natural clusters in predictors.

\par To focus our discussions, we consider two prediction algorithms, namely high dimensional linear least squares and random forest regressions, to be compared through their margin of differential behavior in a  merging versus ensembling framework. The choice of these two particular methods is carefully guided by previous research which we briefly discuss next. In particular,   \cite{ramchandran2021} methodologically highlighted the efficacy of ensembling over merging for Random Forest learners trained on data containing clusters. Across a wide variety of situations, including a variety of data distributions, number and overlap of clusters, and outcome models, they found the ensembling strategy to produce remarkable improvements upon merging --  at times even producing an over 60\% reduction in prediction error. The ensembling method similarly produced powerful results on cancer genomic data containing clusters, highlighting its potential applications for real datasets with feature distribution heterogeneity. Conversely, \textcolor{black}{they} observed that linear regression learners with a small number of features produced no significant differences between ensembling and merging when the model was correctly specified, even in the presence of fully separated clusters within the training set. These interesting observations obviate analytical investigation of where the benefits in ensembling for Random Forest arise, and why the same results are not achieved with linear regression.

\par In this paper, we explore the role of the bias-variance interplay in conferring the benefits of ensembling. We show that for unbiased high-dimensional linear regression, even optimally weighted ensembles do not asymptotically improve over a single regression trained on the entire dataset -- and in fact perform significantly worse when the dimension of the problem grows proportional to sample size. Conversely, we show that for ensembles built from Random Forest learners (which are biased for a linear outcome model), ensembling is strictly more accurate than merging, regardless of the number of clusters within the training set. We additionally verify our theoretical findings through numerical explorations. 


\par 
We shall also utilize the following language convention for the rest of the paper. By Single Cluster Learners (SCL's), we will refer to any supervised learning algorithm that can produce a prediction model using a single cluster. In this paper, we consider linear regression and random forest as two representative SCL's for reasons described above. The term \textit{Ensemble}, to be formally introduced in Section \ref{sec:math_formalization}, will indicate training an SCL on each cluster within the training set and then combining all cluster-level predictions using a specific weighting strategy, creating a single predictor that can be applied to external studies. The \textit{Merged} method, also to be formally introduced in Section \ref{sec:math_formalization}, will refer to the strategy of training the same single learning algorithm chosen to create the SCL's on the entire training data. With this language in mind, the rest of the paper is organized in two main subsections pertaining
to comparing the results of merging versus ensembling 
from both theoretical and numerical perspectives, for linear least regression (Section \ref{sec:least_square}) and random forest regression (Section \ref{sec:random_forest}) respectively.

\section{Main results}\label{sec:main_results}
We divide our discussions of the main results into the following subsections. In section \ref{sec:math_formalization} we introduce a mathematical framework under which our analytic explorations will be carried out. Subsequently, sections \ref{sec:least_square} and \ref{sec:random_forest} provide our main results on comparison between ensembling and merging methods through the lens of linear least squares and random forest predictors respectively. These subsections also contain the numerical experiments to verify and extend some of the intuitions gathered from the theoretical findings.

\subsection{Mathematical Formalization}\label{sec:math_formalization}
We consider independent observations $(Y_i,\mathbf{x}_i), i=1,\ldots,n$ on $n$  individuals with $Y_i\in \mathbb{R}$ and $\mathbf{x}_i\in \mathbb{R}^p$  denoting  outcome of interest and $p$-dimensional covariates respectively. Our theoretical analyses will be carried out under a linear regression assumption on the conditional distribution of $Y$ given $\mathbf{x}$ as follows: $Y_i=\mathbf{x}_i^T\boldsymbol{\beta}+\boldsymbol{\varepsilon}_i, i=1,\ldots,n,\label{eqn:model}$ where $\boldsymbol{\varepsilon}_i$ are i.i.d. random variables independent of $\mathbf{x}_i$ with mean $0$ and variance $\sigma^2$. Although the conditional outcome distributions remains identical across subjects, we will naturally introduce sub-structures in the covariates $\mathbf{x}_i$'s as found through subject matter knowledge and/or pre-processing through clustering. To this end, given any $K\geq 2$ we consider a partition of $\{1,\ldots,n\}$ into disjoint subsets $ \mathbb{S}_1,\ldots,\mathbb{S}_K$ with $|\mathbb{S}_t|=n_t, t\geq 1$ and $\sum_t n_t=n$. Subsequently, we define the $t^{\rm th}$ data sub-matrix as $(\bY_t:\bbX_t)_{n_t\times (p+1)}$ with $\bY_t=(Y_i)_{i\in \mathbb{S}_t}$ and $\bbX_t$ collecting corresponding $\mathbf{x}_i$'s in its rows for $i\in \mathbb{S}_t$. We also denote the merged data matrix as $(\bY,\bbX)$ where $\bY=(\bY_1^T \ldots \bY_K^T)$ and $\bbX=(\bbX_1^T \ldots \bbX_K^T)$. This setup allows us to define ensembling single-cluster learners and merged prediction strategies as follows.

\begin{algorithm}[ht]    
\caption{: \textbf{Ensembling}}
\begin{algorithmic}[1]
    \State for $t = 1,\ldots ,K$:
        \begin{itemize}
            \item Compute $\hat{Y}_t(\bx_{\star})$, the prediction of the SCL trained on $(\bY_t:\bbX_t)$ at new point $\bx_{\star}$
            \item Determine $\hat{w}_t$, the weight given to $\hat{Y}_t(\bx_{\star})$ within the ensemble, using chosen weighting scheme
        \end{itemize}
        \State The prediction made by the \textit{Ensemble} as a function of the weight vector  $\mathbf{w}=({w}_1,\ldots,{w}_K)$: $\hat{Y}_{{\mathbf{w}},E}(\bx_{\star}) = \sum_{t = 1}^K \hat{w}_t \hat{Y}_t(\bx_{\star})$;
\Algphase{Algorithm 2 : Merging}
    \begin{enumerate}
        \item Train same learning algorithm used for the SCLs on $(\bY,\bbX)$
        \item The prediction of this \textit{Merged} learner on $\bx_{\star}$ is $\hat{Y}_M(\bx_{\star})$
    \end{enumerate}
\end{algorithmic}
  \end{algorithm}


The rest of the paper will focus on analyzing linear least squares and random forest regressions under the above notational framework. As we shall see, the main difference in the behavior of these two SCL's arise due to their respective underlying bias variance interplay. This in turn necessitates different strategies for understanding ensembling methods based on these two learners. In particular, since it has been shown in [\cite{ramchandran2021}, Figure 5] that the bias comprises nearly the entirety of the MSE for both \textit{Merged} and \textit{Ensemble} learners constructed with random forest SCLs and trained on datasets containing clusters, we focus on the behavior of the squared bias for the random forest regression analysis. In contrast, for the analysis of least squares under model \eqref{eqn:model}, the lack of any bias of the least squares method implies that one needs to pay special attention to the weighting schemes used in practice. To tackle this issue, we first discuss asymptotic behavior of the commonly used stacking \citep{breiman1996stacked} weighting method (-- see e.g. \citep{Ramchandran2020,ramchandran2021} and references therein for further details) through the lens of numerical simulations. Building on this characterization, we subsequently appeal to tools of random matrix theory to provide a precise asymptotic comparison between ensembling and merging techniques.

\subsection{Linear Least Squares Regression}\label{sec:least_square}

We now present our results on linear least squares regression under both ensembling and merging strategies. To this end, we shall always assume that each of $\bbX_t$ and $\bbX$ are full column rank almost surely. In fact, this condition holds whenever $p\leq \min_t n_t$ and the $\mathbf{x}_i$'s have non-atomic distributions \citep{eaton1973non}. Other "nice" distributions additionally adhere to such a condition with high probability in large sample sizes, and thus the main ideas of our arguments can be extended to other applicable distributional paradigms. If we consider a prediction for a new data point $\bx_{\star}$, the \textit{Merged} and \textit{Ensemble} predictors based on a linear least squares SCL can be written as 
\begin{align*}
\hat{Y}_{M}&=\bx_{\star}^T(\bbX^T\bbX)^{-1}\bbX^T\bY,\\
    \hat{Y}_{\hat{\mathbf{w}},E}&=\sum_{t=1}^K{\hat{w}_l\bx_{\star}^T\left(\bbX_t^T\bbX_t\right)^{-1} \bbX_t^T \bY_t},
\end{align*}
where $\hat{w}_l$ are the stacking weights described in section 2 of \cite{ramchandran2021}. Stacked regression weights, originally proposed by Leo Breiman in 1996 \citep{breiman1996stacked}, 
have been shown in practice to produce optimally performing ensembles \cite{Stacking1996, Ramchandran2020}. Therefore, a first challenge in analyzing the performance of the ensemble predictor $\hat{Y}_E$ is handling the data-dependent nature of the stacking weights $\hat{w}_t$'s. To circumvent this issue we refer to our experimental results (see e.g. Table \ref{table:invstack} in Section \ref{sec:simulation_lin_reg}) to verify that asymptotically, the stacking weights described in \cite{Patil2018} represent a convex weighting scheme. In particular, we remark that across a variety of simulations, the stacking weights corresponding to $K$ clusters asymptotically lie on the $K$-simplex. A closer look reveals that the weights individually converge to ``universal" deterministic limits which sum up to $1$. Consequently, in the rest of this section we shall work with stacking weights upholding such requirements. We note that we will not aim to derive a formula (although this can be done with some work), but rather will simply demonstrate that even the optimal weighting for predicting the outcome for a new point $\bx_{\star}$ is inferior compared to the the merged predictor $\hat{Y}_M$. This will immediately imply the benefits of using the $\hat{Y}_M$ instead of $\hat{Y}_E$ in the case of linear least squares regression under a linear outcome model.  
We next provide a characterization of the optimal convex combination of weights for predicting an outcome corresponding to the new point $\bx_{\star}$ as inverse variance weighting (IVW) in the following lemma. 


\begin{lemma}
\label{lemma:inv_variance_weighting}
Denote by $\mathbf{w}$ the $K$-dimensional vector of weights, ${\mathbf{w}}=[{w}_1, \ldots, {w}_K]^T$. Then, the inverse variance weighting (IVW) scheme is the solution to the following optimization problem

\begin{align*}
    \min_{\mathbf{w}} \mathrm{Var}(\hat{Y}_E({\mathbf{w}},\bx_{\star})) \quad \text{ s.t.} \quad \sum_{t = 1}^K {w}_t = 1 
\end{align*}
yielding ${w}_t^{\rm opt} = \left(\sum_{t = 1}^K \frac{1}{\sigma_t^2} \right)^{-1} \frac{1}{\sigma_t^2}$ for $t = 1, \ldots, K$, with $\sigma^2_t=\mathrm{Var}\left(\bx_{\star}^T\left(\bbX_t^T\bbX_t\right)^{-1}\bbX_t^T\bbY_t\right)$.
\end{lemma}

Using the results from Lemma \ref{lemma:inv_variance_weighting}, we can now provide a comparison of the mean squared prediction error of the \textit{Merged} predictor $\hat{Y}_M$ with the optimally weighted \textit{Ensemble} predictor 
$\hat{Y}_{w,\mathrm{opt}}(\bx_{\star}):=\sum_{j=1}^Kw_{t}^{\mathrm{opt}}(\bx_{\star})\bx^T_{\star}\left(\bbX_t^T\bbX_t\right)^{-1} \bbX_t^T\bY_t$ from input $\bx_{\star}$ under a high dimensional asymptotics where $\frac{p}{n}\rightarrow \gamma \in [0,1)$. 
It is simple to see that both learners are unbiased, and therefore the MSE's are driven by the following variances:
\begin{align}
    \text{Var}\left[ \hat{Y}_{w,\rm opt}(\bx_{\star}) |\bx_{\star}\right] &= \bigg[\sum_{t = 1}^K \left[\bx_{\star}^T(\X_t^T\X_t)^{-1}\bx_{\star}\right]^{-1}\bigg]^{-1}\\
    \text{Var}\left[ \hat{Y}_M(\bx_{\star})|\bx_{\star} \right] &= \bx_{\star}^T\bigg[\sum_{t = 1}^{K}\X_t^T\X_t\bigg]^{-1}\bx_{\star}
\end{align}

\begin{theorem}\label{theorem:highdim}
Suppose $K=2$, $\lambda_t=n/n_t$, $\mathbf{x}_i\sim N(\bmu_1,I)$ for $i\in \mathbb{S}_1$ and $\mathbf{x}_i\sim N(\bmu_2,I)$ for $i\in \mathbb{S}_2$, and $\bx_{\star}\sim N(\mathbf{0},I)$. Assume that $p, n \to \infty$ such that $p/n\rightarrow \gamma$ and $p/n_t\rightarrow \lambda_t\gamma<1$ for $t\in \{1,2\}$. Also assume that $\bmu_1,\bmu_2$ are uniformly bounded (in $n,p$) norm. Then, the following holds:  
\begin{align*}
     \frac{\mathrm{Var}\left[\hat{Y}_M(\bx_{\star})|\bx_{\star}\right]}{\mathrm{Var}\left[ \hat{Y}_{w,\rm opt}(\bx_{\star})|\bx_{\star}\right]} &\rightarrow  
     \frac{\gamma}{1-\gamma}\times {\sum_{t=1}^2\frac{1-\lambda_t\gamma}{\lambda_t\gamma}}=\frac{1-2\gamma}{1-\gamma}\quad \text{in probability.}
\end{align*}
\end{theorem}

\par A few comments are in order regarding the assumptions made and subsequent implications of the derived results. First, we note that assumptions are not designed to be optimal but rather to present a typical instance where such a comparison holds. For example, the assumption of normality and spherical variance-covariance matrix is not necessary and can be replaced easily by $\mathbf{x}_i\sim \bmu_t+\Sigma^{1/2}\bZ_i$ for $\bZ_i$ having i.i.d. mean-zero, variance $1$ coordinates with bounded $8^{\rm th}$ moments, and $\Sigma$ having empirical spectral distribution converging to a compact subset of $\mathbb{R}_+$ -- (see e.g. \cite{bai2008large}). 
Moreover, the mean $\mathbf{0}$ nature of the $\bx_{\star}$ also allows some simplification of the formulae the arise in the analyses and can be extended to incorporate more general means by appealing to results of quadratic functionals of the Green's Function of $\mathbb{X}_t^T\mathbb{X}_t$ \citep{mestre2006asymptotic}. 
Similarly, the normality of $\bx_{\star}$ also allows ease of computation for higher order moments of quadratic forms of $\bx_{\star}$ and can be done without modulo suitable moment-assumptions. We do not pursue these generalizations for the sake of space and rather focus on a set-up which appropriately conveys the main message of the comparison between the \textit{Ensemble} and \textit{Merged} predictors. Moreover, we have also taken $p<\min_t n_t$ to produce exactly unbiased prediction through least squares. However, our analyses tools can be easily extended to explore ridge regression in $p>n$ scenarios (i.e. $\gamma>1$ in the context of Theorem \ref{theorem:highdim}). Finally, we note that the $K=2$ assumption is also made for simplification of proof ideas and one can easily conjecture a general $K$ result from our proof architecture (i.e. a ratio of $(1-K\gamma)/(1-\gamma)$ asymptotically).

As for the main message, we can conclude that Theorem \ref{theorem:highdim} provides evidence that the asymptotic prediction mean squared of the \textit{Merged} is strictly less than that of the \textit{Ensemble} for positive values of $\gamma$. Furthermore, for $\gamma = 0$, indicating all slower growing dimensions $p$ compared to $n$, the ratio in Theorem \ref{theorem:highdim} converges to 1, indicating that the two approaches are asymptotically equal in this case and providing further theoretical evidence to support the empirical findings. 

We end this subsection by presenting the results for the fixed dimensional $p$ case, in order to demonstrate that the above results are not merely a high dimensional artifact.  The next theorem, the proof of which is standard and provided in the Supplementary Materials, provides further support to this theory.

\begin{theorem}\label{theorem: linreg_merged_vs_ensemble}
Suppose $K=2$ with $\mathbf{x}_i\sim N(\bmu_1,I)$ for $i\in \mathbb{S}_1$ and $\mathbf{x}_i\sim N(\bmu_2,I)$ for $i\in \mathbb{S}_2$ with $\|\bmu_{\mathbf{1}}\| = \|\bmu_{\mathbf{2}}\| = 1$, and that $\bx_{\star}$ is randomly drawn from a distribution with mean $\mathbf{0}$ and variance covariance matrix $I$. If $p=O(1)$ as $n\rightarrow \infty$, then there exists $\kappa<1$ such that

\begin{align*}
    \frac{\mathrm{Var}\left[\hat{Y}_M(\bx_{\star})|\bx_{\star}\right]}{\mathrm{Var}\left[ \hat{Y}_{w,\rm opt}(\bx_{\star})|\bx_{\star}\right]} &\rightarrow \kappa<1, \quad \text{in probability}
\end{align*}
\end{theorem}
Once again, the assumptions made are not designed to be optimal but rather to present a typical instance where such a comparison holds. In terms of the result,
Theorem \ref{theorem: linreg_merged_vs_ensemble} provides analytic evidence that for two clusters under linear least squares regression under a fixed dimensional linear model, the \textit{Merged} asymptotically produces predictions that are strictly more accurate than the \textit{Ensemble}. This is indeed congruent with the empirical results. Although we do not provide exact analytic arguments for more clusters (since they result in  more complicated expressions in the analyses without changing the main theme of results), the same relationship persists in simulations as the number of clusters increases (see e.g. Table \ref{table:invstack} displays results for $K = 5$). Overall, we once again conclude that when both the \textit{Merged} and \textit{Ensemble} are unbiased by using linear least squares predictions, there is no benefit to cluster-specific ensembling -- even when the dimension is fixed.

\subsection{Simulations}\label{sec:simulation_lin_reg}

\begin{table}[ht]
\centering
\vspace*{0mm}\hspace*{0cm}\includegraphics[scale = .54]{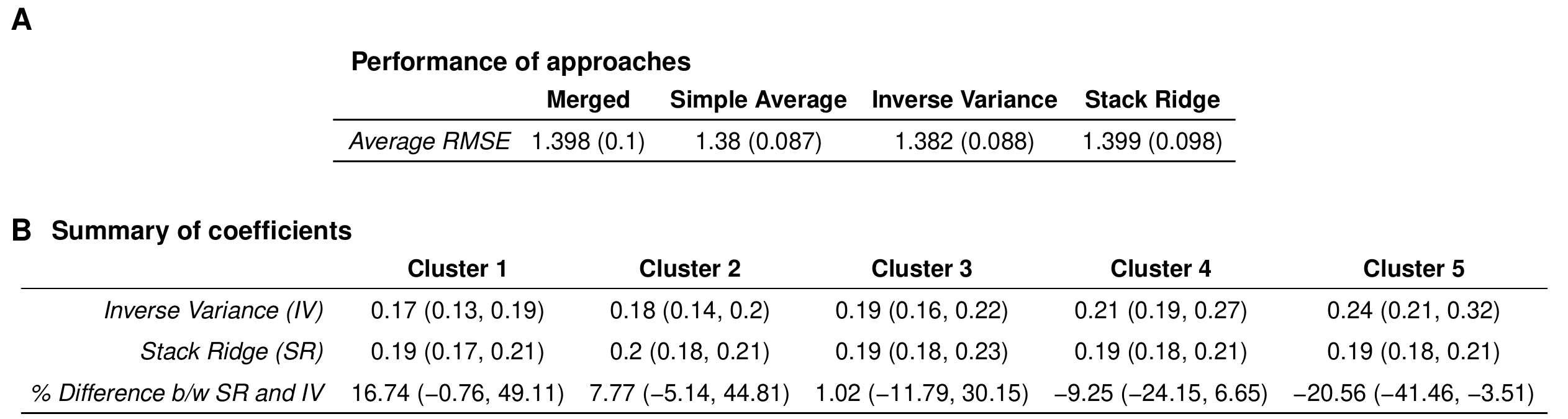}
\caption{Performance and ensemble coefficient values for linear regression SCLs trained on 5 simulated Gaussian clusters per iteration, over 100 reps.  (A) Average RMSE of the four different ensembling methods on test sets. Standard deviations are shown in parentheses. (B) Ensemble coefficient values for the 5 SCLs in training, ranked by Inverse Variance (IV) weighting. 95\% confidence intervals are shown in parentheses. }
\label{table:invstack}
\end{table}

We first present our numerical experiments to demonstrate asymptotic behavior of the stacking weights in Table \ref{table:invstack}. In the simulation used to create Table \ref{table:invstack}, ensembles of least square regression learners combined with IVW or stacking were compared when trained and tested on datasets generated by gaussian mixture models. For each of 100 iterations, a training set with 5 clusters was generated using the {\tt clusterGeneration} package in R, with median values of between-cluster separation \cite{clusterGeneration}. Test sets were drawn using the same general approach with 2 clusters. All outcomes were generated through a linear model, following an analogous framework to \cite{ramchandran2021}. To form the ensembles, least squares regression learners were first trained on each cluster; then, all SCLs were combined through either stacking or inverse variance weights, the latter for which each SCL was weighted proportionally to the inverse sample variance of its prediction accuracy on all clusters not used in training.

\par From Table \ref{table:invstack}A, we observe that there is no significant difference in prediction accuracy between the \textit{Merged} learner with ensembles weighted through simple averaging, IVW, or stacking. Table \ref{table:invstack}B displays the weights given by either IVW or stacking to each cluster within training, with clusters ordered from the lowest to the highest weight ranges. The distribution of weights for both approaches are centered around the same value, as evidenced by the median cluster weight for Cluster 3. Stacking weights essentially mimic simple averaging, while IVW in general results in a slightly larger range of weights. However, the equal prediction performance of ensembles constructed using either weighting scheme demonstrates that these slight differences in the tails of the weighting distribution produce negligible effects. Simple averaging, IVW, and stacking weights are all centered around the same value, indicating that each SCL on average is able to learn the true covariate-outcome relationship to a similar degree. Furthermore, the \textit{Merged} learner is able to achieve the same level of accuracy, illustrating that at least empirically, there is no benefit to ensembling over merging for least squares regression SCLs.

We next present our numerical experiments to demonstrate the accuracy of Theorem \ref{theorem:highdim} in the $p,n$ asymptotic regime. \begin{figure}[ht]
  \centering
  \includegraphics[width=.7\linewidth]{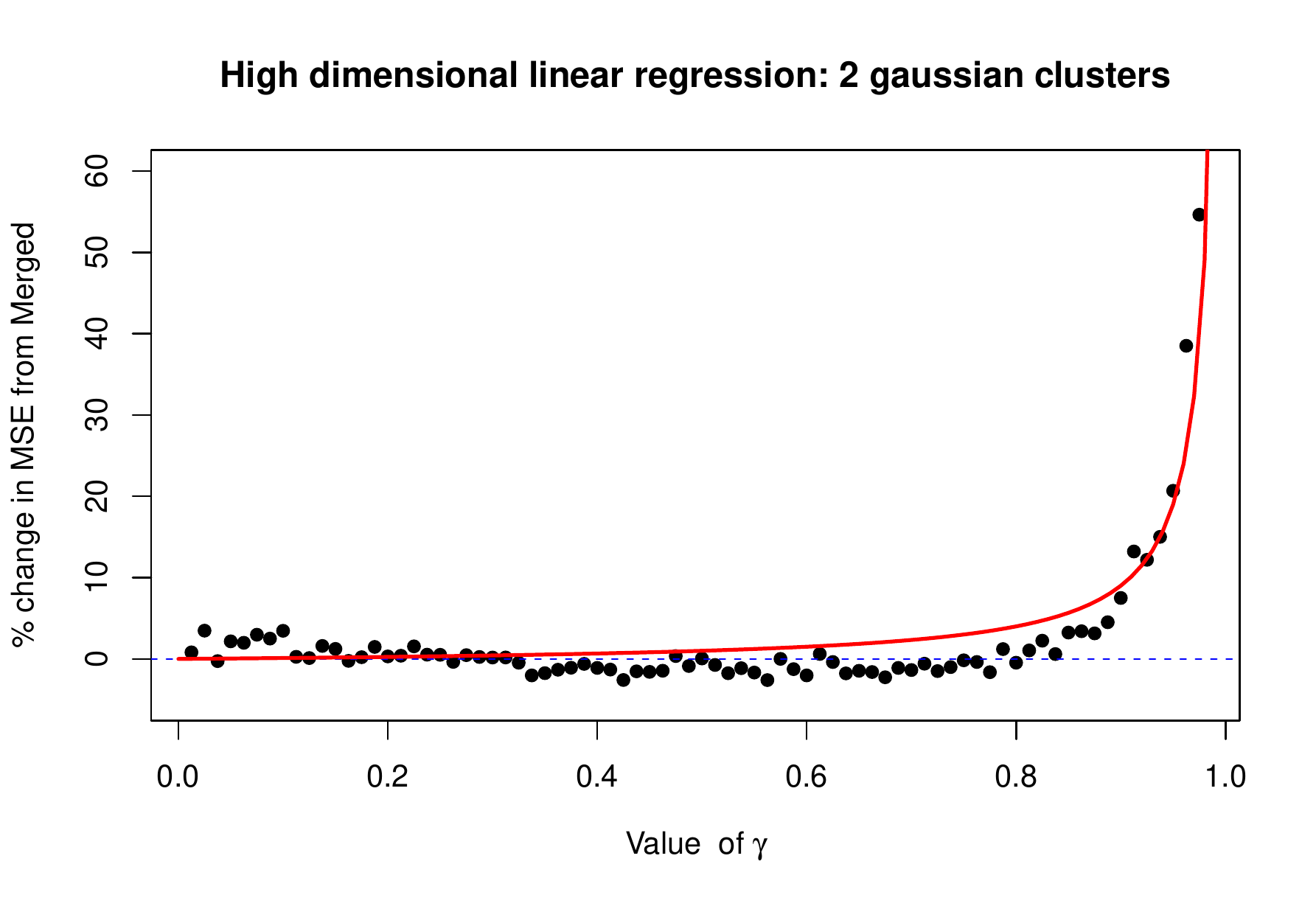}
  \caption{Percent change in average MSE of ensembling approaches compared to the \textit{Merged} for linear regression \textit{Ensemble} learners trained and tested on datasets with 2 equal sized gaussian clusters, as a function of $\gamma_t = \frac{p}{n_t}$. The theoretical limit is shown in red. }
  \label{fig:F1}
\end{figure}
In this regard, Figure \ref{fig:F1} presents the percent change in the average MSE of the \textit{Ensemble} compared to the \textit{Merged} using both the limiting expressions presented in Theorem \ref{theorem:highdim} and results from a simulation study conforming to all assumptions. In the simulation, we set the number of samples $n_t$ per cluster at 400 for $t = 1, 2$, and varied the dimension $p$ incrementally from 0 to 400 to evaluate the performance of these approaches for different values of $\gamma_t = \frac{p}{n_t}$. Clusters were again simulated using the {\tt clusterGeneration} package as described above, and for the sake of simplicity, the coefficients selected to simulate the linear outcome for each cluster were set to be opposites of one another for the first and second clusters. From Figure \ref{fig:F1}, we observe that the theoretical and empirical results closely mirror one another, and that the performance of the \textit{Ensemble} and the \textit{Merged} are almost identical for values of $\gamma_t < .8$, after which the \textit{Merged} produces an exponential comparative increase in performance. These results together with the fixed dimension asymptotics presented in Theorem \ref{theorem: linreg_merged_vs_ensemble} show that for unbiased linear regression, it is overall more advantageous to train a single learner on the entire dataset than to ensemble learners trained on clusters.

\subsection{Random Forest Regression}\label{sec:random_forest}
Next, we examine the asymptotic risk of the \textit{Merged} and \textit{Ensemble} approaches built with random forest SCLs. It has previously been shown that for regression tasks on clustered data, the bias dominates the MSEs of random forest-based learners \cite{ramchandran2021}. Additionally, these results indicate that reduction of the bias empirically constitutes the entirety of the improvement in prediction performance of the \textit{Enemble} over the \textit{Merged} (see figure 5 of \cite{ramchandran2021}). As we have just discussed in the above section, there is no advantage to ensembling over merging for unbiased least squares regression, and in fact the \textit{Merged} learner is asymptotically strictly more accurate than the \textit{Ensemble}. Therefore, we can explore random forest as an SCL in order to determine the effect of ensembling over merging for learning approaches that are biased for the true outcome model, and furthermore pinpoint whether bias reduction is the primary mechanism through which the general cluster-based ensembling framework produces improvements for forest-based learners. 
\par One of the greatest challenges in developing results for random forest is the difficulty in analyzing Breiman's original algorithm; therefore, we will build our theory upon the centered forest model initially proposed by Breiman in a technical report \cite{Breiman2004}. The primary modification in this model is that the individual trees are grown independently of the training sample, but it still has attractive features such as variance reduction through randomization and adaptive variable selection that characterize the classic random forest algorithm. Previous work has shown that if the regression function is sparse, this model retains the ability of the original to concentrate the splits only on the informative features; thus, in this section, we introduce possible sparsity into our considered outcome model. Throughout, we will be using notation consistent with \cite{klusowski2020} and \cite{biau2012}, and will advance upon the work of the former on upper bounding the bias of centered forests on standard uniform data to now present upper bounds for the bias of the \textit{Merged} and \textit{Ensemble} for centered forest SCLs trained on datasets containing clusters. 

\subsubsection{Training sets with two uniformly distributed clusters}
We begin by considering a simple yet representative data-generating model corresponding closely to that on which several previous analytical work in this area has been based \citep{biau2012,klusowski2020}. In particular, we shall assume throughout this section that $p$ is fixed, but will still keep track the effect of the number of non-zero coefficients $S$ in the leading term of the bias of the ensembled and merged predictors. In our analyses, we consider a training set consisting of two non-overlapping, uniform clusters. in particular, the distributions of the two training clusters are given by $\X_1 \stackrel{\rm i.i.d.}{\sim} [\mathrm{U}\left(0, \frac{1}{2}\right)]^p$ and $\X_2 \stackrel{\rm i.i.d.}{\sim} [\mathrm{U}(1, \frac{3}{2})]^p$, with $n_1 = n_2$; that is, each row consists of $p$ independent uniform variables, with the specific distributional parameters depending on cluster membership. We shall often express each training observation shortly as (with an abuse of notation)
\begin{align*}
    \bx_{i} &\stackrel{\rm i.i.d.}{\sim} \left[\mathrm{U}\left(0, \frac{1}{2}\right)\right]^p\mathbbm{1}\{i \in \mathbb{S}_1\} + \left[\mathrm{U}\left(1, \frac{3}{2}\right)\right]^p\mathbbm{1}\{i \in \mathbb{S}_2\} \text{ for } i = 1,...,n
\end{align*}
As before, we denote by $\mathbb{S}_1$ and $\mathbb{S}_2$ the respective set of indices in first and second clusters. Now, we shall consider the distribution of the new point $\bx_{\star}$ is a deterministic mixture of the first and second clusters; that is, $\bx_{\star} \sim \rm A \times [\mathrm{U}\left(0, \frac{1}{2}\right)]^p + (1- \rm A) \times [\mathrm{U}(1, \frac{3}{2})]^p$, where $\rm A \sim $ $\text{Bernoulli}\left(\frac{1}{2}\right)$, representing an random indicator of cluster membership of $\bx_{\star}$. We note here that the choice of parameters for each cluster-level distribution (i.e. the end points of the uniform cluster) are arbitrary; our calculations are simply dependent on the width of the interval within each uniform and that the ranges are non-overlapping. 

\par

In this section our analysis will also be able to incorporate certain sparsity structures of $\bbeta$ in \eqref{eqn:model}. In particular, given any subset $\mathbf{S}\subset \{1,\ldots,p\}$, we suppose that $\bbeta$ in our outcome model has non-zero coordinates restricted to $\mathbf{S}$. The outcome model can therefore be described by the regression function $f(\bx) = \bx_{\mathbf{S}}^T\bbeta_S$ -- the  vector $\bx_{\mathbf{S}}$ therefore denoting the covariates corresponding to only the $S$ strong features captured in vector $\bbeta_S$ out of $p$ total features. Let $S=|\mathbf{S}|$ equal the number of 'strong' features, indicating the covariates that have non-zero contributions to the outcome in the linear relationship.  Thus, $f(\bx)$ is a sparse regression function, with the degree of sparsity mediated by $S$. This underlying model implies a convenient asymptotic representation of the random forest development scheme as described below. 

\par Now, using the language from \cite{klusowski2020}, we define $p_{nj}$ as the probability that the $j^{th}$ variable is selected for a given split of tree construction, for $j = 1,..,p$. In an idealized scenario, we are aware of which variables are strong - in this case, the ensuing random procedure produces splits that asymptotically choose only the strong variables with probability $1/S$ and all other variables with zero probability; that is, $p_{nj} \to  1/S$ if $j \in \mathbf{S}$ and 0 otherwise. The splitting criteria, as described by [\cite{biau2012}, Section 3], is as follows: at each node of the tree, first select $M$ variables with replacement; if all chosen variables are weak, then choose one at random to split on. Thereafter, one can argue that under suitable assumptions,  centered forests asymptotically adaptively select only the strong features, and therefore we can intuitively restrict all further analysis to those $S$ variables. In this regard, we will assume hereafter that $p_{nj}=\frac{1}{S}(1+\xi_{nj})$ for $j\in \mathbf{S}$ and $=\xi_{nj}$ otherwise -- where for each $j$,  $\xi_{nj}$ is a sequence that tends to $0$ as $n\rightarrow \infty$.  Finally, we shall let $p_n=\frac{1}{S}(1+\min_{j}\xi_{jn})$. Also, we let $\log_2k_n$ denote the number of times the process of splitting is repeated for parameter $k_n > 2$ as defined in \cite{klusowski2020}.  

\par The above randomized splitting scheme allows us to formalize the definition of a centered random forest based on a given training data $(Y_i,\mathbf{x}_i)_{i\in \mathcal{D}_n}$ as follows. Given 
a new test point $\bx_{\star}$, and randomizing variable $\theta$ that defines the probabilistic mechanism building each tree, we define $A_n(\bx_{\star}, \theta)$ as the box of the random partition containing test point $\bx_{\star}$ and the individual tree predictor as
\begin{align*}
    f_n(\bx_{\star};\theta, \mathcal{D}_n) &= \frac{\sum_{i \in \mathcal{D}_n } Y_i \mathbbm{1}\{\bx_{i} \in A_n(\bx_{\star}, \theta)\}}{\sum_{i \in \mathcal{D}_n} \mathbbm{1}\{ \bx_{i} \in A_n(\bx_{\star}, \theta )\}} \mathbbm{1}\{\epsilon_n(\bx_{\star}, \theta)\} 
\end{align*}
where $\epsilon_n(\bx_{\star}, \theta\}$ is the event that $\sum_{i = 1} \mathbbm{1}\{\bx_{i} \in A_n(\bx_{\star}, \theta\} > 0$; that is, there is at least one training point that falls within the same partition as $\bx_{\star}$. The forest represents an average over all such trees; we can then obtain the prediction made by the forest by taking the expectation of the individual tree predictors with respect to the randomizing variable $\theta$: 
\begin{align*}
\bar{f}_n(\bx_{\star};\theta, \mathcal{D}_n) &= \sum_{i \in \mathcal{D}_n} Y_i \mathbbm{E}_{\theta}\left[\frac{\mathbbm{1}\{\bx_{i} \in A_n(\bx_{\star}, \theta\}}{\sum_{i \in \mathcal{D}_n} \mathbbm{1}\{\bx_{i} \in A_n(\bx_{\star}, \theta\}} \mathbbm{1}\{\epsilon_n(\bx_{\star}, \theta)\} \right]
\end{align*}
\par Now, we note that the \textit{Ensemble} approach performs a weighted average of the prediction of two forests, one trained on each cluster. As before, let $\hat{Y}_1(\bx_{\star})$, $\hat{Y}_2(\bx_{\star})$, and $\hat{Y}_E(\bx_{\star})$ designate the respective predictions of the forests trained on clusters 1, 2, and the overall ensemble on test point $\bx_{\star}$. For the centered forest algorithm, coordinates of the test point that lie outside the range of the training data for that forest are given predictive values of 0. Therefore, only the forest trained on the first cluster gives non-zero predictions on coordinates in the interval $[0, \frac{1}{2}]$ and vice versa for the forest trained on the second cluster and coordinates in the interval $[1, \frac{3}{2}]$. For convenience of notation and ease of proof architecture, we internally weight each learner in the overall ensemble by cluster membership of $\bx_{\star}$ within the functional forms of $\hat{Y}_1(\bx_{\star})$ and $\hat{Y}_2(\bx_{\star})$. 
  Using this representation, we provide precise expressions for the predictions made by the Merged and Ensemble learners on $\bx_{\star}$ during the proof of our main theorem in the supplementary material -- which essentially considers a weighting scheme based on $w_1=\mathbbm{1}(\bx_{\star}\in [0,1/2]^p )$ and  $w_2=\mathbbm{1}(\bx_{\star}\in [1,3/2]^p )$. We show that both predictions can be represented as weighted averages of the training outcome values, with the choice of weights as the differentiating factor between the two methods.

\par With this, we are now ready present our main result for the mean squared prediction errors of the \textit{Ensemble} and \textit{Merged} approaches under the centered random forest approach described above. At this point we appeal to numerical evidence demonstrating that the squared bias dominates the variance in asymptotic order [\cite{ramchandran2021}, Figure 5] and thereby focus on the asymptotics of the bias in this article. 
The variance component has been empirically and theoretically shown to converge to zero faster than the squared bias, and thus analysis of the squared bias term asymptotically approximates the MSE \cite{klusowski2020}. Our next result reports the leading term of the squared bias, where we will utilize the assumptions on $p_{nj}$, $p_n$, and $k_n$ discussed above. 
\begin{theorem} 
\label{theorem: rf_uniform}
Assume that $\bbeta_{\mathbf{S}} \sim N(\bzero,  I_{S\times S})$  , $K = 2$, $\mathbf{x}_i \stackrel{\rm i.i.d.}{\sim} [\mathrm{U}\left(0, \frac{1}{2}\right)]^p$ for $i\in \mathbb{S}_1$, and $\mathbf{x}_i \stackrel{\rm i.i.d.}{\sim} [\mathrm{U}(1, \frac{3}{2})]^p$ for $i\in \mathbb{S}_2$. Additionally suppose that $p_{nj}\log k_n \to \infty$ for all $j = 1, \ldots p$ and $k_n/n \to 0$ as $n \to \infty$. Then, 
\begin{enumerate}
    \item [(i)] The squared bias of the \textit{Ensemble} is upper bounded by
    \begin{align*}
    \frac{\text{S}}{8} k_n^{log_2(1 - 3p_n/4)}\left[1 + o(1) \right]
\end{align*}
    \item [(ii)] The squared bias of the \textit{Merged} is upper bounded by 
    \begin{align*}
        \frac{\text{S}}{4} k_n^{log_2(1 - 3p_n/4)}\left[1 + o(1) \right]
    \end{align*}
\end{enumerate}
\end{theorem}

We first note the similarity of these bound to the results in \cite{klusowski2020} (and therefore the bound converges to $0$ owing to the assumption of $p_{n,j}\log{k_n}\rightarrow \infty$) and thereafter note that the upper bound for the \textit{Merged} is exactly twice as high as the corresponding quantity for the \textit{Ensemble}. The bounds for both learners depend only on the number of non-sparse variables S, parameter $k_n$, and split probability $p_n$, with the squared bias terms converging at the same rate $O(k_n^{log_2(1 - 3p_n/4)})$. Intuition as to why the \textit{Merged} does not perform as well as the \textit{Ensemble} in this case may be gleaned from the following observation: the \textit{Merged} has the same bound as would be achieved by a single forest trained on observations arising from a $[U(a, a + 1)]^p$ distribution for some constant $a$. That is, the \textit{Merged} ignores the cluster-structure of the data and instead treats the data as arising from the average of its' component distributions, whereas the \textit{Ensemble} produces improvements by explicitly taking the cluster-structure into account. 

\subsubsection{Simulations}\label{sec:simulations}
In this subsection we verify that the intuitions gathered above from the theory considering uniform distribution of covariates within each clusters also extend to other distributions in numerical experiments.
\begin{figure}[ht]
  \centering
  \includegraphics[width=5in]{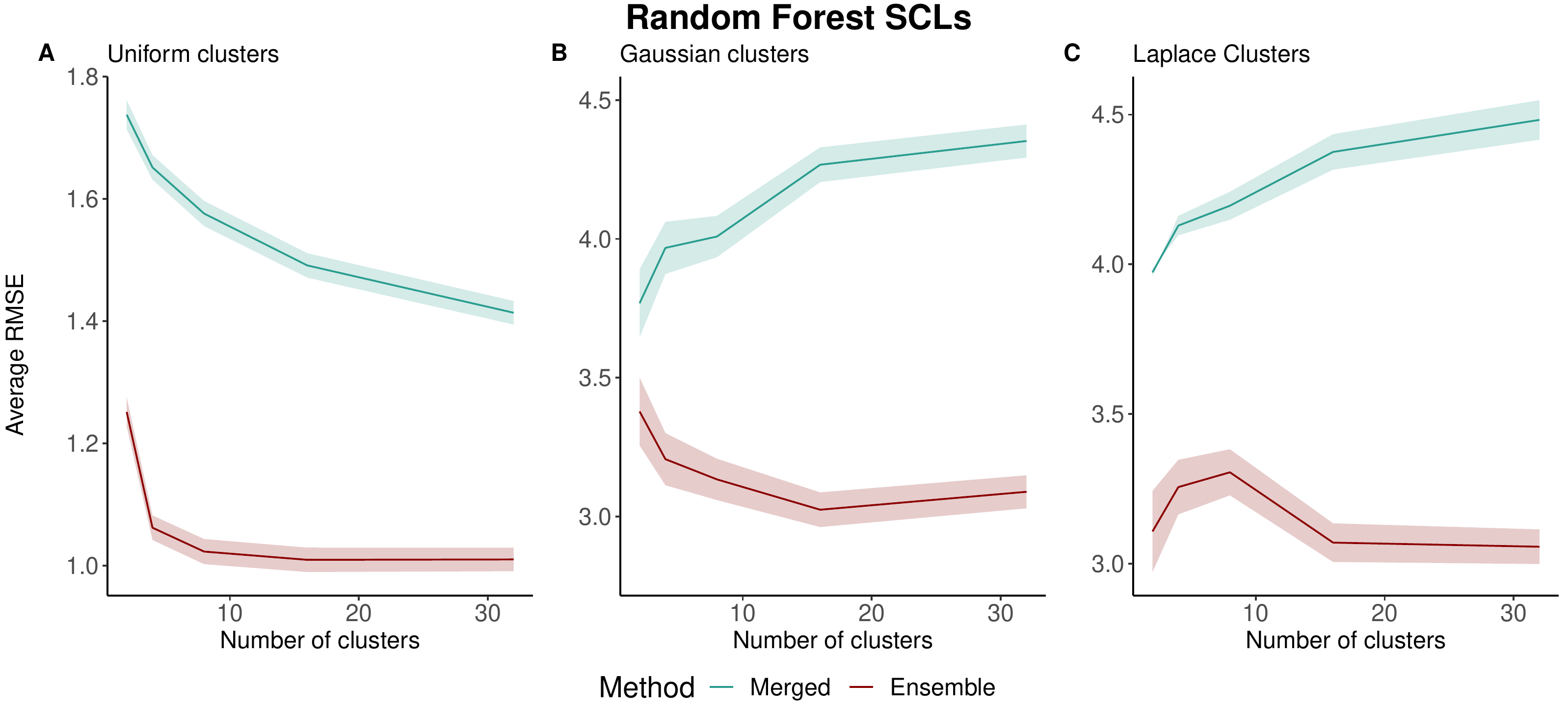}
  \caption{Average RMSE of the \textit{Merged} and the \textit{Ensemble} as a function of the number of clusters in the training set. \textbf{(A)} Uniform clusters \textbf{(B)} Multivariate Gaussian clusters \textbf{(C)} Multivariate Laplace-distributed clusters}
  \label{fig:F2}
\end{figure}
In this regard, Figure \ref{fig:F2} displays the results of simulation studies to experimentally validate the frameworks presented in Theorem \ref{theorem: rf_uniform} and Theorem 5 
in Section A.8 of the Appendix for clustered data arising from three different distributional paradigms. We observe that regardless of distribution, the \textit{Ensemble} produces a significantly lower average RMSE than the \textit{Merged}, and that the difference between the two methods levels out as $K$ increases. Interestingly, across all distributions considered, the \textit{Ensemble} empirically converges to produce about a 33\% improvement over the \textit{Merged} for high $K$, whereas the theoretical upper bounds illustrate a continual increase of the magnitude of improvement in an exponential manner. This indicates that while the relationship between the upper bounds derived in Theorem 5 
are confirmed through the simulations, the theoretical bounds get significantly less tight for higher $K$.

\section{Discussion}\label{sec:discussions}
In this paper, we have provided first steps towards a theoretical understanding of possible benefits and pitfalls of ensembling learners based on clusters within training data compared to ignoring the grouping structure of the covariates. Our result suggests that the benefits vary depending on the nature of underlying algorithm and often plays out differently based on the mutual dominance between the bias and variance. We verify some of the numerical observations made in practice and synthetic data using the cases of linear regression and random forest regression -- each chosen to present to contrasting perspectives. This represents a first effort into providing a theoretical lens to this phenomenon. Several questions remain - in particular, an exploration of the exact asymptotics of the stacking mechanism while ensembling, as well as incorporating actual data driven clustering algorithms (such as k-means) as a pre-processing step, are both interesting open directions worth future study.
\clearpage 
\bibliographystyle{plainnat}
\bibliography{bibliography}

\appendix

\section{Appendix}
Code to reproduce simulations may be found at \href{https://github.com/m-ramchandran/ensembling-vs-merging}{https://github.com/m-ramchandran/ensembling-vs-merging}.
\subsection{Proof of Lemma 1}
Inverse variance weighting is a method of combining multiple random variables to minimize the variance of the weighted average. In the case of linear regression learners, the SCL-level predictions have closed form-variances  $\text{Var}\left[\hat{Y}_t(\bx_{\star})\right] =  \sigma_t^2$ for $t = 1,...,K$. Thus, $\text{Var}(\hat{Y}_E) = \sum_{t = 1}^k {w}_t^2 \sigma_t^2$ such that ${w}_t = 1$ to ensure un-biasedness of the resulting ensemble. 

\begin{proof}
To minimize $\text{Var}(\hat{Y}_E)$ subject to the constraint that $\sum_{t = 1}^K {w}_t = 1$, we use the method of Lagrange multipliers, and express the variance as: 
\begin{align*}
    \text{Var}(\hat{Y}_E) &= \sum_{t = 1}^k {w}_t^2 \sigma_t^2 - a\left(\sum_{t = 1}^K {w}_t - 1\right)
\end{align*}
We minimize by finding the root of this equation when $t > 0$;
\begin{align*}
    0 &= \frac{\partial}{\partial {w}_t} Var(\hat{Y}_E) \\
    &= 2  {w}_t  \sigma_t^2 - a\\
    \Rightarrow {w}_t &= \frac{a/2}{\sigma_t^2}
\end{align*}
Now, using the fact that $\sum_{t = 1}^K {w}_t  = 1$, 
\begin{align*}
    \frac{2}{a} &= \sum_{t = 1}^K \frac{1}{\sigma_t^2}
\end{align*}

Therefore, the individual weights that minimize the variance of the weighted average predictor are:
\begin{align*}
    {w}_{t,\rm opt} &= \left(\sum_{t = 1}^K \frac{1}{\sigma_t^2} \right)^{-1} \frac{1}{\sigma_t^2}
\end{align*}
\end{proof}

\subsection{Proof of Theorem 1}
We start by recalling the formulae for the variances of $\hat{Y}_M$ and $\hat{Y}_{opt}(\bx_{\star})$ for $K=2$
\begin{align}
    \text{Var}\left[ \hat{Y}_{w,\rm opt}(\bx_{\star}) |\bx_{\star}\right] &= \bigg[\sum_{t = 1}^2 \left[\bx_{\star}^T(\X_t^T\X_t)^{-1}\bx_{\star}\right]^{-1}\bigg]^{-1}\label{eqn:cond_var_ens}
\end{align}
\begin{align}
    \text{Var}\left[ \hat{Y}_M(\bx_{\star})|\bx_{\star} \right] &= \bx_{\star}^T\bigg[\sum_{t = 1}^{2}\X_t^T\X_t\bigg]^{-1}\bx_{\star}\label{eqn:cond_var_mer}
\end{align}
Also,  we will write $n_1/n:=1/\lambda_1$ and $n_2/n:=1/\lambda_2=1-1/\lambda_1$ (since $K=2$). We shall now derive individual limits of $\bx_{\star}^T(\X_t^T\X_t)^{-1}\bx_{\star}$ for $t=1,2$ and $\bx_{\star}^T(\sum_{t=1}^2\X_t^T\X_t)^{-1}\bx_{\star}$. 

To this end we shall crucially use the following two lemmas. The first lemma will be used to prove concentration of quadratic forms of $\bx_{\star}$ and is taken from \cite[Lemma B.26]{bai2010spectral}.

\begin{lemma}\label{lemma:quadratic_form_concentration_bai_silverstein}

Let $\mathbf{z}_{\star}=(z_1,\ldots,z_p)\in \mathbb{R}^p$
be a random vector with i.i.d. mean zero entries , for which $\E((\sqrt{p}z_i)^2)=\sigma^2$ and $\sup_{i}\E(|\sqrt{p}z_i|^{4+\eta})<C\infty$ for some $\eta>0$. Moreover, let $A_p$
be a sequence of random $p\times p$ symmetric matrices independent of $\mathbf{z}_{\star}$, with uniformly bounded
eigenvalues. Then the quadratic forms $\mathbf{z}_{\star}^TA_p\mathbf{z}_{\star}$ concentrate around their means at the following
rate
$$\P\left(|\mathbf{z}_{\star}^TA_p\mathbf{z}_{\star}-\frac{\sigma^2}{p}\mathrm{Trace}(A_p)|>\epsilon\right)\leq C_{\eta,\epsilon}p^{-1-\eta/4},$$
where for any $\epsilon,\eta>0$, $C_{\eta,\epsilon}>0$ depends only on $\eta,\epsilon$.
\end{lemma}

Our next lemma will be repeatedly used to justify a lower bound on minimum eigenvalue of $\bbX^T\bbX$ and $\bbX_t^T\bbX_t$ for $t=1,\ldots,K$.

\begin{lemma}\label{lemma:minimum_eigenvalue}
Under the assumptions of Theorem 1, there exists $c>0$ such that with probability converging to $1$ one has $\min_{t=1}^K\left\{\lambda_{\min}(\bbX_t^T\bbX_t/n_t)\right\}\footnote{\text{For any symmetric matrix $A$ we let $\lambda_{\min}(A)$ and $\lambda_{\max}(A)$ denote the minimum and maximum eigenvalues respectively}}\wedge \lambda_{\min}(\bbX^T\bbX/n)\geq c$.
\end{lemma}
The proof of this lemma essentially follows from the Bai-Yin Theorem \cite{bai2008limit} after minor adjustments for the non-zero mean terms $\bmu_t$'s. We present a sketch of a simple argument here for the sake of immediate reference. First note that 

For individual $\lambda_{\min}(\bbX_t^T\bbX_t)$  first  by writing each $\mathbf{x}_i=\mathbf{Z}_i+\bmu_t$ for $i\in \mathbb{S}_t$ (where $\mathbf{Z}_i$ has i.i.d. coordinates of mean-zero and variance $1$ and bounded moments of all order) we have $\mathbb{Z}_t^T=[\mathbf{Z}_1:\cdots:\mathbf{Z}_{n_t}]^T$, $\bar{\mathbf{Z}_t}=\frac{1}{n_t}\sum_{i=1}^{n_t}\mathbf{Z}_i$, and $\hat{v}_t=(\bar{\mathbf{Z}_t}+\mu_t)$. We shall henceforth repeatedly use the fact that  the spectral behavior of $\frac{1}{n_t}\mathbb{Z}_t\mathbb{Z}_t^T-{\bar{\mathbf{Z}_t}}{\bar{\mathbf{Z}_t}^T}$ and $\frac{1}{n_t}\mathbb{Z}_t\mathbb{Z}_t^T$ are asymptotically equivalent (see e.g. \cite[Theorem 2.23]{bloemendal2016principal}).  
In particular,
\begin{align}
    \frac{\bbX_t^T\bbX_t}{n_t}=\tilde{\Sigma}_t+\hat{v}_t\hat{v}_t
    \label{eqn:sherman_morrison_sigma_t}
\end{align},
and therefore 
\begin{align*}
    \lambda_{\min}\left(\frac{\bbX_t^T\bbX_t}{n_t}\right)\geq \lambda_{\min}(\tilde{\Sigma}_t)\stackrel{a.s.}{\rightarrow} \left(1-\frac{1}{\sqrt{\lambda_2\gamma}}\right)^2,
\end{align*}
by \cite[Theorem 2.23]{bloemendal2016principal}. As for $\bbX^T\bbX$ the result immediately follows since it can be written as a sum of individual $\bbX_{t}^T\bbX_t$'s. This completes the proof of Lemma \ref{lemma:minimum_eigenvalue}. 

With Lemma \ref{lemma:quadratic_form_concentration_bai_silverstein} and Lemma \ref{lemma:minimum_eigenvalue} we are now ready to go back to the proof of Theorem 1. Essentially, one immediately have by \ref{lemma:quadratic_form_concentration_bai_silverstein} that for any sequence of random matrices $\hat{A}_{n,p}$ independent from $\bx_{\star}$ and having uniformly bounded operator with high probability for large $n,p$ 
\begin{align*}
    |\frac{1}{p}\bx_{\star}^T\hat{A}_{n,p}\bx_{\star}-\frac{1}{p}\mathrm{Trace}(\hat{A}_{n,p})|\stackrel{\P}{\rightarrow} 0.
\end{align*}
just by defining $\mathbf{z}_{\star}=\frac{\bx_{\star}}{\sqrt{p}}$. Moreover $\left(\frac{\bbX_t^T\bbX_t}{n_t}\right)^{-1}$ and $\left(\frac{\bbX^T\bbX}{n}\right)^{-1}$ have bounded operator norm with probability converging to $1$ and are independent of $\bx_{\star}$. Therefore, coming back to the terms in the conditional variance term of the prediction error (in \eqref{eqn:cond_var_ens}) and \eqref{eqn:cond_var_mer}), we have
\begin{align}
    \bx_{\star}^T\left(\bbX_t^T\bbX_t\right)^{-1}\bx_{\star}-\frac{\lambda_t\gamma}{p}\mathrm{Trace}\left(\frac{\bbX_t^T\bbX_t}{n_t}\right)^{-1}=\frac{p}{n_t}\frac{1}{p}\bx_{\star}^T\left(\frac{\bbX_t^T\bbX_t}{n_t}\right)^{-1}\bx_{\star}-\frac{\lambda_t\gamma}{p}\mathrm{Trace}\left(\frac{\bbX_t^T\bbX_t}{n_t}\right)^{-1}\stackrel{\P}{\rightarrow}0. \label{eqn:cond_var_cond_conc_ens}
\end{align}
The application of Lemma \ref{lemma:quadratic_form_concentration_bai_silverstein}  in the above argument comes simply by defining $\mathbf{z}_{\star}=\frac{\bx_{\star}}{\sqrt{p}}$. Similarly, one also has
\begin{align}
    \bx_{\star}^T\left(\bbX^T\bbX\right)^{-1}\bx_{\star}-\frac{\gamma}{p}\mathrm{Trace}\left(\frac{\bbX^T\bbX}{n}\right)^{-1}=\frac{p}{n}\frac{1}{p}\bx_{\star}^T\left(\frac{\bbX^T\bbX}{n}\right)^{-1}\bx_{\star}-\frac{\gamma}{p}\mathrm{Trace}\left(\frac{\bbX_t^T\bbX_t}{n_t}\right)^{-1}\stackrel{\P}{\rightarrow}0.\label{eqn:cond_var_cond_conc_mer}
\end{align}
Therefore in the remaining proof we simply focus on deriving the limits of $\frac{\lambda_t\gamma}{p}\mathrm{Trace}\left(\frac{\bbX_t^T\bbX_t}{n_t}\right)^{-1}$ for $t=1,\ldots,K$ and $\frac{\gamma}{p}\mathrm{Trace}\left(\frac{\bbX^T\bbX}{n}\right)^{-1}$. To this first note that by Sherman-Morrison Woodbury identity applied to \eqref{eqn:sherman_morrison_sigma_t} we have
\begin{align}
   \frac{\lambda_t\gamma}{p}\mathrm{Trace}\left(\frac{\bbX_t^T\bbX_t}{n_t}\right)^{-1}&= \frac{\lambda_t\gamma}{p}\mathrm{Trace}\left[\tilde{\Sigma}_t^{-1}-\frac{\tilde{\Sigma}_t^{-1}\hat{v}_t\hat{v}_t^T\tilde{\Sigma}_t^{-1}}{1+\hat{v}_t^T\tilde{\Sigma}_t^{-1}\hat{v}_t}\right]. \label{eqn:sherman_morrison_xt}
\end{align}
Now we have that the empirical spectral distribution of $\tilde{\Sigma}_t$ converges weakly almost surely to the standard Marchenko-Pastur distribution (see e.g. \cite[Theorem 2.23]{bloemendal2016principal}) with density $f_{t,\rm MP}:=\frac{1}{2\pi}\frac{\sqrt{(\gamma_{t,+}-x)(x-\gamma_{t,-})}}{\lambda_t\gamma x}$ with $\gamma_{t,\pm}=(1\pm\sqrt{\lambda_t\gamma})^2$ (and also the smallest eigenvalue converges to something strictly positive almost surely). Thereafter using the steps of calculations described in the proof of Proposition 2 of \cite{hastie2019surprises} we have

\begin{align}
    \frac{\lambda_t\gamma}{p}\mathrm{Trace}\left(\tilde{\Sigma}_t^{-1}\right)\stackrel{\P}{\rightarrow} \frac{\lambda_t\gamma}{1-\lambda_t\gamma}, \quad t=1,\ldots,K.\label{eqn:marchenko_pastur_xt}
\end{align}
Moreover, 
\begin{align*}
    \left\vert\frac{\lambda_t\gamma}{p}\mathrm{Trace}\left(\frac{\tilde{\Sigma}_t^{-1}\hat{v}_t\hat{v}_t^T\tilde{\Sigma}_t^{-1}}{1+\hat{v}_t^T\tilde{\Sigma}_t^{-1}\hat{v}_t}\right)\right\vert\leq \frac{\lambda_t\gamma}{p}\hat{v}_t^T\tilde{\Sigma}_t^{-2}\hat{v}_t\leq \frac{\lambda_t\gamma\|\hat{v}_t\|_2^2\|\tilde{\Sigma}_t^{-2}\|_{\rm op}}{p}
\end{align*},
where we have used $\|\cdot\|_{\rm op}$ to denote the operator norm of a matrix. Now we have by Lemma \ref{lemma:minimum_eigenvalue} followed by \cite[Theorem 2.23]{bloemendal2016principal} that
$\|\tilde{\Sigma}_t^{-2}\|_{\rm op}=O_{\P}(1)$ (i.e. stays bounded in probability) as $n,p\rightarrow \infty$. Moreover $\|\hat{v}\|_t^2\leq 2\|\bar{\mathbf{Z}}_t\|_2^2+2\|\bmu_t\|_2^2=O_{\P}(1)$ since $\bmu_t$'s have uniformly bounded norms by assumptions of the theorem and by standard Markov's Inequality for $\|\bar{\mathbf{Z}}_t\|_2^2$ because this can be written as $\sum_{j=1}^p\left(\frac{1}{n_t}\sum_{i=1}^{n_t}Z_{ij}\right)^2$ which equals a sum of $p$-independent $\chi_1^2/n_t$ random variables and hence have mean $p/n_t\rightarrow \lambda_t\gamma$ and variance $2p/n_t^2\rightarrow 0$ (and thereby implying order $1$ behavior). Therefore, we have $\|\hat{v}_t\|_2^2\|\tilde{\Sigma}_t^{-2}\|_{\rm op}=O_{\P}(1)$ as $n_t,n,p\rightarrow \infty$ according to the stipulations of Theorem 1. This immediately implies that $\frac{\lambda_t\gamma}{p}\mathrm{Trace}\left(\frac{\tilde{\Sigma}_t^{-1}\hat{v}_t\hat{v}_t^T\tilde{\Sigma}_t^{-1}}{1+\hat{v}_t^T\tilde{\Sigma}_t^{-1}\hat{v}_t}\right)\stackrel{\P}{\rightarrow} 0$ and therefore by \eqref{eqn:sherman_morrison_xt}, \eqref{eqn:marchenko_pastur_xt}, and \eqref{eqn:marchenko_pastur_xt} we have that
\begin{align}
    \bx_{\star}(\bbX_t^T\bbX_t)^{-1}\bx_{\star}\stackrel{\P}{\rightarrow}\frac{\lambda_t\gamma}{1-\lambda_t\gamma}, \ t=1,\ldots,K.\label{eqn:pred_limit_xt} \\
\end{align}
Next we turn to $ \bx_{\star}(\bbX^T\bbX)^{-1}\bx_{\star}$ and one use \eqref{eqn:cond_var_cond_conc_mer} and can prove using the ideas in \cite{benaych2016spectral} that $\bx_{\star}(\bbX^T\bbX)^{-1}\bx_{\star}\rightarrow \frac{\gamma}{1-\gamma}$. Here for the sake of clarity and intuition why this result is true, we provide a simple calculation which demonstrates the result for $K=2$, $n_1=n_2=n/2$, and small enough $\gamma>0$ and $\|\bmu_1\|_2,\|\bmu_2\|_2$ (i.e. the argument below goes through for some $c>0$ whenever $\gamma>0$ and $\|\bmu_1\|_2,\|\bmu_2\|_2$ are all less than $c>0$ and for more general $c>0$ an argument along the lines of \cite{benaych2016spectral} yields the desired result). In particular, using these simplifying assumptions and the notation $\hat{\Sigma}=\sum_{i=1}^n\mathbf{Z}_i\mathbf{Z}_i^T$ we have by some algebra that
\begin{align*}
    (\bbX^T\bbX/n)&=\hat{\Sigma}+\frac{1}{2}(\bar{\mathbf{Z}}_1+\bmu_1)(\bar{\mathbf{Z}}_1+\bmu_1)^T+\frac{1}{2}(\bar{\mathbf{Z}}_2+\bmu_2)(\bar{\mathbf{Z}}_1+\bmu_1)^T-\frac{1}{4}\left(\bar{\mathbf{Z}}_1-\bar{\mathbf{Z}}_2\right)\left(\bar{\mathbf{Z}}_1-\bar{\mathbf{Z}}_2\right)^T\\
    &=\hat{\Sigma}+\hat{U}D\hat{U}^T
\end{align*}
where $\hat{U}=[(\bar{\mathbf{Z}}_1+\bmu_1):(\bar{\mathbf{Z}}_2+\bmu_2):(\bar{\mathbf{Z}}_1-\bar{\mathbf{Z}}_2)]$ and $D=\mathrm{diag}(1/2,1/2,-1/4)$. Therefore we have by Woodbury identity that
\begin{align*}
\frac{\gamma}{p}\mathrm{Trace}(\bbX^T\bbX/n)^{-1}=\frac{\gamma}{p}\mathrm{Trace}\left[\hat{\Sigma}^{-1}-\hat{\Sigma}^{-1}\hat{U}(D^{-1}+\hat{U}^T\hat{\Sigma}^{-1}\hat{U})^{-1}\hat{U}^T\hat{\Sigma}^{-1}\right].
\end{align*}
Next note that
\begin{align*}
   | \frac{1}{p}\mathrm{Trace}\left[\hat{\Sigma}^{-1}\hat{U}(D^{-1}+\hat{U}^T\hat{\Sigma}^{-1}\hat{U})^{-1}\hat{U}^T\hat{\Sigma}^{-1}\right]|&=| \frac{1}{p}\mathrm{Trace}\left[(D^{-1}+\hat{U}^T\hat{\Sigma}^{-1}\hat{U})^{-1}\hat{U}^T\hat{\Sigma}^{-2}\hat{U}\right]|
\end{align*}
Now since $\lambda_{\min}(\hat{\Sigma}^{-1})\rightarrow (1-\sqrt{\gamma})^2$ in probability, for any $\delta>0$, there exists $c_{\delta}>0$ such that $\gamma,\|\bmu_1\|_2,\|\bmu_2\|_2\leq c_{\delta}$ implies that $\|\hat{U}^TD^{-1}\hat{U}\|_{\rm op}<\frac{1}{4}-\delta$ with probability larger than $1-\delta$ (this comes with direct computation and the concentration of $\|\bar{\mathbf{Z}_t}\|_2^2$ around $\lambda_t\gamma_t$). Also, by operator norm bound and concentration of the norms of the columns of $\hat{U}$ we have $\frac{1}{p}\mathrm{Trace}\left[\hat{\Sigma}^{-1}\hat{U}(D^{-1}+\hat{U}^T\hat{\Sigma}^{-1}\hat{U})^{-1}\hat{U}^T\hat{\Sigma}^{-1}\right]\stackrel{\P}{\rightarrow} 0$. Therefore, 
\begin{align*}
    \frac{\gamma}{p}\mathrm{Trace}(\bbX^T\bbX/n)=\frac{\gamma}{p}\mathrm{Trace}(\hat{\Sigma}^{-1})+o_{\mathbb{P}}(1)\stackrel{\P}{\rightarrow}\frac{\gamma}{1-\gamma}.
\end{align*}
The last line above follows by direct calculations with the limiting Marchenko-Pastur law of the empirical spectral distribution of $\hat{\Sigma}$ (see e.g. the calculations \cite{hastie2019surprises}). This completes the proof for small enough $\gamma>0$ and $\|\bmu_1\|_2,\|\bmu_2\|_2$ and balanced two clusters. A similar proof can be extended to general $K$ clusters without the simplifications for the two balanced cluster case can be obtained for small enough values of $\gamma>0$ and $\|\bmu_t\|_2$'s as follows. 

We elaborate on an initial surgery first. To this end we first claim that one can write
\begin{align}
    \frac{\gamma}{p}\mathrm{Trace}\left(\frac{\bbX^T\bbX}{n}\right)^{-1}&=\frac{\gamma}{p}\mathrm{Trace}\left[\frac{1}{n}\sum_{i\in \cup_{t=1}^K\mathbb{S}_t}\mathbf{Z}_i\mathbf{Z}_i^T+\hat{U}\hat{V}^{T}\right]^{-1},
    \label{eqn:sherman_morrison_sigma}
\end{align}
where recall that we write $\bX_i=\mathbf{Z}_i\bmu_t$ for $i\in \mathbb{S}_t$ and the pair $(U_{p\times 3K},V_{p \times 3K})$ can be defined as follows: for $t=0,\ldots,K-1$ define $U_{\cdot,3t+1}=U_{\cdot,3t+2}=\frac{1}{\sqrt{\lambda_t}}\mu_{t+1}^T$ and $U_{\cdot,3(t+1)}=\frac{1}{\sqrt{\lambda_t}}\bar{\mathbf{Z}}_{t+1}^T$, and $V^T_{3t+1,\cdot}=V_{3(t+1),\cdot}=\frac{1}{\sqrt{\lambda_t}}\mu_{t+1}^T$ and $V_{3t+2,\cdot}=\frac{1}{\sqrt{\lambda_t}}\bar{\mathbf{Z}}_{t+1}^T$ \footnote{For any $m_1\times m_2$ matrix $A$ we let $A_{\cdot,j}$ and $A_{i,\cdot}$  denote its $j^{\rm th}$ column and $i^{\rm th}$ row respectively for $i=1,\ldots,m_1$ and $j=1,\ldots, m_2$}. Applying Woodbury Identity to \eqref{eqn:sherman_morrison_sigma} and realizing that $\frac{1}{n}\sum_{i\in \cup_{t=1}^K\mathbb{S}_t}\mathbf{Z}_i\mathbf{Z}_i^T:=\hat{\Sigma}$ is the variance-covariance mattrix of a standard $n\times p$ Gaussian matrix, we have by Woodbury's identity that
\begin{align*}
     \frac{\gamma}{p}\mathrm{Trace}\left(\frac{\bbX^T\bbX}{n}\right)^{-1}&=
     \frac{\gamma}{p}\mathrm{Trace}\left[\hat{\Sigma}^{-1}-\hat{\Sigma}^{-1}\hat{U}\left\{I+\hat{V}^T\hat{\Sigma}^{-1}\hat{U}\right\}^{-1}\hat{V}^T\hat{\Sigma}^{-1}\right]\\
     &=\frac{\gamma}{p}\mathrm{Trace}(\hat{\Sigma}^{-1})-\gamma\frac{\mathrm{Trace}\left\{\left(I+\hat{V}^T\hat{\Sigma}^{-1}\hat{U}\right)^{-1}\hat{V}^T\hat{\Sigma}^{-2}\hat{U}\right\}}{p}
\end{align*}
where we have used the fact that $\hat{\Sigma}$ is invertible (by \cite{eaton1973non} and discussion above). Now one again for small enough values of $\gamma>0$ and $\|\bmu_t\|_2$'s
\begin{align*}
   \left\vert\frac{\mathrm{Trace}\left\{\left(I+\hat{V}^T\hat{\Sigma}^{-1}\hat{U}\right)^{-1}\hat{V}^T\hat{\Sigma}^{-2}\hat{U}\right\}}{p}\right\vert \stackrel{\P}{\rightarrow} 0,
\end{align*}
by arguments similar to the $K=2$ case. The rest of the proof follows similarly. Therefore, we have presented the proofs also for the general $K$ case albeit for small enough values of $\gamma>0$ and $\|\bmu_t\|_2$'s. The most general proof allowing for more general values of $\gamma>0$ and $\|\bmu_t\|_2$'s can be obtained by appealing to \cite{benaych2016spectral}.


\subsection{Lemma \ref{lemma: randomness}}
\begin{lemma}\label{lemma: randomness}
Assume that $n_t$ is the number of samples in the $t^{th}$ cluster $\mathbb{S}_t$ and $\E(\mathbf{x}_i)=\bmu_t$ and $\mathrm{Var}(\mathbf{x}_i)=I$ for $i\in \mathbb{S}_t$ such that $\frac{n_t}{n}-1/\lambda_t=o(1)$ for $t=1,\ldots,K$.  Also assume that $\bx_{\star}$ is normal with mean zero and variance $I$. Then the following hold form any finite $p$ as $n\rightarrow \infty$
\begin{enumerate}
    \item [(i)] 
    \begin{align*}
      \frac{\mathrm{Var}\left[ \hat{Y}_{w,\rm opt}(\bx_{\star})|\bx_{\star}\right] }{d_1}
    &\stackrel{\P}{\rightarrow} 1 .
    \end{align*}
    \item [(ii)]
    \begin{align*}
     \frac{\mathrm{Var}\left[ \hat{Y}_M(\bx_{\star}) |\bx_{\star}\right]}{d_2}
     &\stackrel{\P}{\rightarrow} 1.
\end{align*}
\end{enumerate}

where 
\begin{align*}
    d_1 &= \bigg[\sum_{t = 1}^k \left(\frac{1}{n_t}\left[(p-1) + \frac{1}{1  + \|\bmu_{\mathbf{t}}\|^2}\right]\right)^{-1} \bigg]^{-1} \\
    d_2 &= \frac{1}{n} \mathrm{Trace} \left( I + \sum_{t = 1}^k \frac{n_t}{n} \bmu_{\mathbf{t}}\bmu_{\mathbf{t}}^T\right)^{-1}
\end{align*}
\end{lemma}

\begin{proof}
\begin{enumerate}
    \item [(i)]
We begin by considering the term $\bx_{\star}^T(\X_t^T\X_t)^{-1}\bx_{\star}$ and note that
\begin{align*}
    \bx_{\star}^T(\X_t^T\X_t)^{-1}\bx_{\star} &= \frac{1}{n_t} \bx_{\star}^T\left(\frac{\X_t^T\X_t}{n^t}\right)^{-1}\bx_{\star}
\end{align*}
Recall that $n_t$ denotes the number of observations in the $t^{th}$ training cluster, and $\X_t = [\bx_{\mathbf{1}}, ..., \bx_{\mathbf{n_t}}]^T$, where $\bx_{\mathbf{i}}$ is a  $p \times 1$  vector for $i =  1,...,n_t$. Then, by Standard CLT
\begin{align*}
    \frac{\X_t^T\X_t}{n^t} &= \frac{1}{n_t}\sum_{i\in \mathbb{S}_t} \bX_{i}^T \bX_{i}+O_{\P}(1/\sqrt{n_t}) \quad \text{(in both Operator and Frobenius norm)}\\
    &= \mathrm{Var}[\bX_{1}] + \E[\bX_{1}]\E[\bX_{1}]^T+O_{\P}(1/\sqrt{n_t}) \\
    &=  I + \bmu_{\mathbf{t}} \bmu_{\mathbf{t}}^T+O_{\P}(1/\sqrt{n_t})
\end{align*}
Therefore, it is easy to see that
\begin{align*}
    \frac{1}{n_t}  \left[ \bx_{\star}^T\left(\frac{\X_t^T\X_t}{n^t}\right)^{-1}\bx_{\star}\right] - \frac{1}{n_t} \bx_{\star}^T (\mathrm{I} + \bmu_{\mathbf{t}} \bmu_{\mathbf{t}}^T)^{-1} \bx_{\star}=O_{\P}(1/\sqrt{n_t^3}).
\end{align*}
In the last line we have used the operator norm inequality to use the fact that $\left(\frac{\X_t^T\X_t}{n^t}\right)^{-1}$ converges to $(\mathrm{I} + \bmu_{\mathbf{t}} \bmu_{\mathbf{t}}^T)^{-1}$ in operator norm, the error is $O_{\P}(1/\sqrt{n_t})$ in operator norm, and the fact that $\frac{\|\bx_{\star}\|_2^2}{n_t}=O_{\mathbb{P}}(1/n_t)$ under our assumptions. We therefore focus on the quantity $\frac{1}{n_t} \bx_{\star}^T (\mathrm{I} + \bmu_{\mathbf{t}} \bmu_{\mathbf{t}}^T)^{-1} \bx_{\star}$ and find its expectation and variance to demonstrate suitable concentration. Also since $\bx_{\star}\sim N(\mathbf{0},I)$ we have $\E\left(\frac{1}{n_t} \bx_{\star}^T (\mathrm{I} + \bmu_{\mathbf{t}} \bmu_{\mathbf{t}}^T)^{-1} \bx_{\star}\right)\sim \frac{1}{n_t} \text{Trace}\left[(\rm I + \bmu_{\mathbf{t}} \bmu_{\mathbf{t}}^T)^{-1}\right]=\Theta(1/n_t)$ (where the notation $\Theta$ defines exact order). Also by direct calculation $\mathrm{Var}\left(\frac{1}{n_t} \bx_{\star}^T (\mathrm{I} + \bmu_{\mathbf{t}} \bmu_{\mathbf{t}}^T)^{-1} \bx_{\star}\right)\sim \frac{1}{n_t^2} \text{Trace}\left[(\rm I + \bmu_{\mathbf{t}} \bmu_{\mathbf{t}}^T)^{-1}\right]^2=O(1/n_t^2)$. This implies that 
\begin{align*}
    \frac{1}{n_t}  \left[ \bx_{\star}^T\left(\frac{\X_t^T\X_t}{n^t}\right)^{-1}\bx_{\star}\right] - \frac{1}{n_t} \text{Trace}\left[(\rm I + \bmu_{\mathbf{t}} \bmu_{\mathbf{t}}^T)^{-1}\right]=O_{\P}(1/\sqrt{n_t^3})
\end{align*}
This completes the proof of the first result.

\item [(ii)] For the proof of this part, first note that by similar argument of as part (i) we have
\begin{align*}
    \mathrm{Var}(\hat{Y}_M(\bx_{\star})|\bx_{\star})=\bx_{\star}^T(\bbX^T\bbX)^{-1}\bx_{\star}=\frac{1}{n}\bx_{\star}\mathrm{Trace}\left(I+\sum_{t=1}^K\lambda_t^{-1}\bmu_t\bmu_t\right)^{-1}\bx_{\star}=o_{\P}(1/n).
\end{align*}
The rest of the proof then follows verbatim as before.

\end{enumerate}
\end{proof}

\subsection{Lemma \ref{lemma: twoclusters}}
We simplify the expected variances of the \textit{Ensemble} and the \textit{Merged} from Lemma \ref{lemma: randomness} in the case in which there are two clusters within the training set.
\begin{lemma}\label{lemma: twoclusters}
Assume $K = 2$. We denote $\tilde\lambda_t = \frac{n_t}{n}$ for $t = 1,  2$; by  construction, $\lambda_1  +  \lambda_2 =  1$. Then, for any finite $p$ we have the following: 
\begin{enumerate}
    \item [(i)]\begin{align*}
     \frac{\mathrm{Var}\left[ \hat{Y}_{w,\rm opt}(\bx_{\star})|\bx_{\star}\right]}{s_1} &\stackrel{\P}{\rightarrow}  1
\end{align*}
    \item [(ii)] \begin{align*} 
     \frac{\mathrm{Var}\left[ \hat{Y}_{M}(\bx_{\star})|\bx_{\star}\right]}{s_2}
     &\stackrel{\P}{\rightarrow} 1
\end{align*}
where 
\begin{align*}
    s_1 &= \frac{1}{n}\left[(p - 1) + \frac{(p-1)[1 +\tilde\lambda_2 \bmu_{\mathbf{1}}^T \bmu_{\mathbf{1}}  + \tilde\lambda_1 \bmu_{\mathbf{2}}^T \bmu_{\mathbf{2}}] + 1}{(p - 1) (1  + \bmu_{\mathbf{1}}^T \bmu_{\mathbf{1}}) (1  + \bmu_{\mathbf{2}}^T \bmu_{\mathbf{2}}) + 1  + \tilde\lambda_1\bmu_{\mathbf{1}}^T \bmu_{\mathbf{1}} + \tilde\lambda_2\bmu_{\mathbf{2}}^T \bmu_{\mathbf{2}}} \right] \\
    s_2 &= \frac{1}{n}\left[p -  \frac{\tilde\lambda_1 \bmu_{\mathbf{1}}^T \bmu_{\mathbf{1}}}{1 + \tilde\lambda_1 \bmu_{\mathbf{1}}^T \bmu_{\mathbf{1}}}  - \frac{\tilde\lambda_2 \bmu_{\mathbf{2}}^T\bmu_{\mathbf{2}} - \frac{2\tilde\lambda_1 \tilde\lambda_2}{1 + \tilde\lambda_1 \bmu_{\mathbf{1}}^T\bmu_{\mathbf{1}}} \bmu_{\mathbf{2}}^T \bmu_{\mathbf{2}} \bmu_{\mathbf{1}}^T \bmu_{\mathbf{1}}  + \frac{ \tilde\lambda_2  \tilde\lambda_1^2}{(1 + \tilde\lambda_1 \bmu_{\mathbf{1}}^T \bmu_{\mathbf{1}})^2}(\bmu_{\mathbf{1}}^T \bmu_{\mathbf{1}})^2}{1 + \tilde\lambda_2  \bmu_{\mathbf{2}}^T \bmu_{\mathbf{2}} - \left(\frac{\tilde\lambda_1 \tilde\lambda_2 }{1 + \tilde\lambda_1 \bmu_{\mathbf{1}}^T \bmu_{\mathbf{1}}}\right) \bmu_{\mathbf{2}}^T \bmu_{\mathbf{1}} \bmu_{\mathbf{1}}^T \bmu_{\mathbf{2}}}\right] 
\end{align*}
\end{enumerate}
\end{lemma}
\begin{proof}
\begin{enumerate}
    \item [(i)]
By Lemma \ref{lemma: randomness}, the limiting $\mathrm{Var}\left[ \hat{Y}_{w,\rm opt}(\bx_{\star})|\bx_{\star}\right]$ when K = 2 is asymptotically equivalent to:
\begin{align*}
    \ & \bigg[n_1\left[(p-1) + \frac{1}{1  + \|\bmu_{\mathbf{1}}\|^2}\right]^{-1} + n_2\left[(p-1) + \frac{1}{1  + \|\bmu_{\mathbf{2}}\|^2}\right]^{-1}\bigg]^{-1} \\
    &= \bigg[n \tilde\lambda_1 \left[(p-1) + \frac{1}{1  + \bmu_{\mathbf{1}}^T \bmu_{\mathbf{1}}}\right]^{-1} + n \tilde\lambda_2 \left[(p-1) + \frac{1}{1  + \bmu_{\mathbf{2}}^T \bmu_{\mathbf{2}}}\right]^{-1}\bigg]^{-1} \\
    &= \frac{1}{n}\left[\frac{\tilde\lambda_1}{(p-1) + \frac{1}{1  + \bmu_{\mathbf{1}}^T \bmu_{\mathbf{1}}}} + \frac{\tilde\lambda_2}{(p-1) + \frac{1}{1  + \bmu_{\mathbf{2}}^T \bmu_{\mathbf{2}}}} \right]^{-1} \\
    &= \frac{1}{n}\left[\frac{\tilde\lambda_1(p-1)  +  \frac{\tilde\lambda_1}{1  + \bmu_{\mathbf{2}}^T \bmu_{\mathbf{2}}} +  \tilde\lambda_2(p-1) + \frac{\tilde\lambda_1}{1  + \bmu_{\mathbf{1}}^T \bmu_{\mathbf{1}}}}{\left[(p-1) + \frac{1}{1  + \bmu_{\mathbf{1}}^T \bmu_{\mathbf{1}}}\right]\times \left[(p-1) + \frac{1}{1  + \bmu_{\mathbf{2}}^T \bmu_{\mathbf{2}}}\right]} \right]^{-1} \\
    &= \frac{1}{n}\left[\frac{\left[(p-1) + \frac{1}{1  + \bmu_{\mathbf{1}}^T \bmu_{\mathbf{1}}}\right]\times \left[(p-1) + \frac{1}{1  + \bmu_{\mathbf{2}}^T \bmu_{\mathbf{2}}}\right]}{\tilde\lambda_1(p-1)  +  \frac{\tilde\lambda_1}{1  + \bmu_{\mathbf{2}}^T \bmu_{\mathbf{2}}} +  \tilde\lambda_2(p-1) + \frac{\tilde\lambda_1}{1  + \bmu_{\mathbf{1}}^T \bmu_{\mathbf{1}}}} \right] 
\end{align*}
The numerator of the above fraction is: 
\begin{align*}
    &\left[(p-1) + \frac{1}{1  + \bmu_{\mathbf{1}}^T \bmu_{\mathbf{1}}}\right]\times \left[(p-1) + \frac{1}{1  + \bmu_{\mathbf{2}}^T \bmu_{\mathbf{2}}}\right]  \\
    &= (p - 1)^2 + \frac{p-1}{1  + \bmu_{\mathbf{2}}^T \bmu_{\mathbf{2}}} + \frac{p-1}{1  + \bmu_{\mathbf{1}}^T \bmu_{\mathbf{1}}} + \frac{1}{(1  + \bmu_{\mathbf{1}}^T \bmu_{\mathbf{1}}) (1  + \bmu_{\mathbf{2}}^T \bmu_{\mathbf{2}})} \\
    &=  \frac{(p - 1)^2(1  + \bmu_{\mathbf{1}}^T \bmu_{\mathbf{1}}) (1  + \bmu_{\mathbf{2}}^T \bmu_{\mathbf{2}}) + (p - 1)[2 + \bmu_{\mathbf{1}}^T \bmu_{\mathbf{1}}  + \bmu_{\mathbf{2}}^T \bmu_{\mathbf{2}}] +  1}{(1  + \bmu_{\mathbf{1}}^T \bmu_{\mathbf{1}}) (1  + \bmu_{\mathbf{2}}^T \bmu_{\mathbf{2}})}\\
    &= \frac{[p - 1]\bigg[(p-1)(1  + \bmu_{\mathbf{1}}^T \bmu_{\mathbf{1}}) (1  + \bmu_{\mathbf{2}}^T \bmu_{\mathbf{2}}) + [1 + \tilde\lambda_1 \bmu_{\mathbf{1}}^T \bmu_{\mathbf{1}}  + \tilde\lambda_2 \bmu_{\mathbf{2}}^T \bmu_{\mathbf{2}}] + [1 + \tilde\lambda_2 \bmu_{\mathbf{1}}^T \bmu_{\mathbf{1}}  +\tilde \lambda_1 \bmu_{\mathbf{2}}^T \bmu_{\mathbf{2}}]\bigg]+ 1}{(1  + \bmu_{\mathbf{1}}^T \bmu_{\mathbf{1}}) (1  + \bmu_{\mathbf{2}}^T \bmu_{\mathbf{2}})}
\end{align*}
The denominator of the above fraction is: 
\begin{align*}
   & \tilde\lambda_1(p-1)  +  \frac{\tilde\lambda_1}{1  + \bmu_{\mathbf{2}}^T \bmu_{\mathbf{2}}} +  \tilde\lambda_2(p-1) + \frac{\tilde\lambda_1}{1  + \bmu_{\mathbf{1}}^T \bmu_{\mathbf{1}}} \\
   &= \frac{(p - 1) (1  + \bmu_{\mathbf{1}}^T \bmu_{\mathbf{1}}) (1  + \bmu_{\mathbf{2}}^T \bmu_{\mathbf{2}}) + \tilde\lambda_1(1  + \bmu_{\mathbf{1}}^T \bmu_{\mathbf{1}}) + \tilde\lambda_2(1  + \bmu_{\mathbf{2}}^T \bmu_{\mathbf{2}})}{ (1  + \bmu_{\mathbf{1}}^T \bmu_{\mathbf{1}}) (1  + \bmu_{\mathbf{2}}^T \bmu_{\mathbf{2}})} \\
   &= \frac{(p - 1) (1  + \bmu_{\mathbf{1}}^T \bmu_{\mathbf{1}}) (1  + \bmu_{\mathbf{2}}^T \bmu_{\mathbf{2}}) + 1  + \tilde\lambda_1\bmu_{\mathbf{1}}^T \bmu_{\mathbf{1}} + \tilde\lambda_2\bmu_{\mathbf{2}}^T \bmu_{\mathbf{2}}}{ (1  + \bmu_{\mathbf{1}}^T \bmu_{\mathbf{1}}) (1  + \bmu_{\mathbf{2}}^T \bmu_{\mathbf{2}})}
\end{align*}
Putting this all together:
\begin{align*}
    & \frac{1}{n}\left[\frac{\left[(p-1) + \frac{1}{1  + \bmu_{\mathbf{1}}^T \bmu_{\mathbf{1}}}\right]\times \left[(p-1) + \frac{1}{1  + \bmu_{\mathbf{2}}^T \bmu_{\mathbf{2}}}\right]}{\tilde\lambda_1(p-1)  +  \frac{\tilde\lambda_1}{1  + \bmu_{\mathbf{2}}^T \bmu_{\mathbf{2}}} +  \tilde\lambda_2(p-1) + \frac{\tilde\lambda_1}{1  + \bmu_{\mathbf{1}}^T \bmu_{\mathbf{1}}}} \right] \\
    &= \frac{1}{n}\left[\frac{[p - 1]\bigg[(p-1)(1  + \bmu_{\mathbf{1}}^T \bmu_{\mathbf{1}}) (1  + \bmu_{\mathbf{2}}^T \bmu_{\mathbf{2}}) + [1 + \tilde\lambda_1 \bmu_{\mathbf{1}}^T \bmu_{\mathbf{1}}  + \tilde\lambda_2 \bmu_{\mathbf{2}}^T \bmu_{\mathbf{2}}] + [1 + \tilde\lambda_2 \bmu_{\mathbf{1}}^T \bmu_{\mathbf{1}}  + \tilde\lambda_1 \bmu_{\mathbf{2}}^T \bmu_{\mathbf{2}}]\bigg]+ 1}{(p - 1) (1  + \bmu_{\mathbf{1}}^T \bmu_{\mathbf{1}}) (1  + \bmu_{\mathbf{2}}^T \bmu_{\mathbf{2}}) + 1  + \tilde\lambda_1\bmu_{\mathbf{1}}^T \bmu_{\mathbf{1}} + \tilde\lambda_2\bmu_{\mathbf{2}}^T \bmu_{\mathbf{2}}}  \right] \\
    &= \frac{1}{n}\left[(p - 1) + \frac{(p-1)[1 + \tilde\lambda_2 \bmu_{\mathbf{1}}^T \bmu_{\mathbf{1}}  + \tilde\lambda_1 \bmu_{\mathbf{2}}^T \bmu_{\mathbf{2}}] + 1}{(p - 1) (1  + \bmu_{\mathbf{1}}^T \bmu_{\mathbf{1}}) (1  + \bmu_{\mathbf{2}}^T \bmu_{\mathbf{2}}) + 1  + \tilde\lambda_1\bmu_{\mathbf{1}}^T \bmu_{\mathbf{1}} + \tilde\lambda_2\bmu_{\mathbf{2}}^T \bmu_{\mathbf{2}}} \right]
\end{align*}

\item [(ii)]

To simplify the expected variance of the \textit{Merged}, we begin by considering the expression $(\rm I + \tilde\lambda_1  \bmu_{\mathbf{1}} \bmu_{\mathbf{1}}^T + \tilde\lambda_2  \bmu_{\mathbf{2}} \bmu_{\mathbf{2}}^T)^{-1}$. To simplify, we use the Sherman-Morrison Inverse formula, which states that for matrix $A_{p \times p}$ and vectors $u, v \in \mathbb{R}^p$, 
\begin{align*}
    (A + uv^T)^{-1} = A^{-1} - \frac{A^{-1}uv^TA^{-1}}{1 + v^TA^{-1}u}
\end{align*}
We apply the Sherman-Morrison formula twice in succession; for the first round, $A = \rm I + \tilde\lambda_1 \bmu_{\mathbf{1}} \bmu_{\mathbf{1}}^T, u = \tilde\lambda_2 \bmu_{\mathbf{2}},$ and $v = \bmu_{\mathbf{2}}$:
\begin{align*}
    (\rm I + \tilde\lambda_1  \bmu_{\mathbf{1}} \bmu_{\mathbf{1}}^T + \tilde\lambda_2  \bmu_{\mathbf{2}} \bmu_{\mathbf{2}}^T)^{-1} &=  (\rm I + \tilde\lambda_1  \bmu_{\mathbf{1}} \bmu_{\mathbf{1}}^T)^{-1}  - \frac{(\rm I + \tilde\lambda_1  \bmu_{\mathbf{1}} \bmu_{\mathbf{1}}^T)^{-1}\tilde\lambda_2 \bmu_{\mathbf{2}} \bmu_{\mathbf{2}}^T (\rm I + \tilde\lambda_1  \bmu_{\mathbf{1}} \bmu_{\mathbf{1}}^T)^{-1}}{1 + \tilde\lambda_2  \bmu_{\mathbf{2}}^T (\rm I + \tilde\lambda_1  \bmu_{\mathbf{1}} \bmu_{\mathbf{1}}^T)^{-1} \bmu_{\mathbf{2}}} 
\end{align*}
We apply the formula again to the term $(\rm I + \tilde\lambda_1  \bmu_{\mathbf{1}} \bmu_{\mathbf{1}}^T)^{-1}$, now, $A = \rm I, u = \tilde\lambda_1 \bmu_{\mathbf{1}}$, and $v = \bmu_{\mathbf{1}}$:
\begin{align*}
    (\rm I + \tilde\lambda_1  \bmu_{\mathbf{1}} \bmu_{\mathbf{1}}^T)^{-1} &= \rm I - \frac{\tilde\lambda_1 \bmu_{\mathbf{1}} \bmu_{\mathbf{1}}^T}{1 + \tilde\lambda_1 \bmu_{\mathbf{1}}^T \bmu_{\mathbf{1}}}
\end{align*}

We plug this into the previous expression:
\begin{align*}
   & (\rm I + \tilde\lambda_1  \bmu_{\mathbf{1}} \bmu_{\mathbf{1}}^T)^{-1}  - \frac{(\rm I + \tilde\lambda_1  \bmu_{\mathbf{1}} \bmu_{\mathbf{1}}^T)^{-1}\tilde\lambda_2 \bmu_{\mathbf{2}} \bmu_{\mathbf{2}}^T (\rm I + \tilde\lambda_1  \bmu_{\mathbf{1}} \bmu_{\mathbf{1}}^T)^{-1}}{1 + \tilde\lambda_2  \bmu_{\mathbf{2}}^T (\rm I + \tilde\lambda_1  \bmu_{\mathbf{1}} \bmu_{\mathbf{1}}^T)^{-1} \bmu_{\mathbf{2}}} \\
    &= \rm I - \frac{\tilde\lambda_1 \bmu_{\mathbf{1}} \bmu_{\mathbf{1}}^T}{1 + \tilde\lambda_1 \bmu_{\mathbf{1}}^T \bmu_{\mathbf{1}}}  - \frac{\left(\rm I - \frac{\tilde\lambda_1 \bmu_{\mathbf{1}} \bmu_{\mathbf{1}}^T}{1 + \tilde\lambda_1 \bmu_{\mathbf{1}}^T \bmu_{\mathbf{1}}}\right)\tilde\lambda_2 \bmu_{\mathbf{2}} \bmu_{\mathbf{2}}^T \left(\rm I - \frac{\tilde\lambda_1 \bmu_{\mathbf{1}} \bmu_{\mathbf{1}}^T}{1 + \tilde\lambda_1 \bmu_{\mathbf{1}}^T \bmu_{\mathbf{1}}}\right)}{1 + \tilde\lambda_2  \bmu_{\mathbf{2}}^T \left(\rm I - \frac{\tilde\lambda_1 \bmu_{\mathbf{1}} \bmu_{\mathbf{1}}^T}{1 + \tilde\lambda_1 \bmu_{\mathbf{1}}^T \bmu_{\mathbf{1}}}\right) \bmu_{\mathbf{2}}} 
\end{align*}
We now take the trace of the above expression, and use the linearity of the trace operator to first consider the trace of each of the component three  terms. We furthermore use the fact that for vector $\nu$, Trace$[\nu \nu^T] = \nu^T\nu$. \\ The first term:
\begin{align*}
    \text{Trace}\left[\rm I\right]  = p
\end{align*}
The second term:
\begin{align*}
    \text{Trace}\left[\frac{\tilde\lambda_1 \bmu_{\mathbf{1}} \bmu_{\mathbf{1}}^T}{1 + \tilde\lambda_1 \bmu_{\mathbf{1}}^T \bmu_{\mathbf{1}}}\right] &= \frac{\tilde\lambda_1}{1 + \tilde\lambda_1 \bmu_{\mathbf{1}}^T \bmu_{\mathbf{1}}}\text{Trace} \left[\bmu_{\mathbf{1}} \bmu_{\mathbf{1}}^T \right] \\
    &= \frac{\tilde\lambda_1 \bmu_{\mathbf{1}}^T \bmu_{\mathbf{1}}}{1 + \tilde\lambda_1 \bmu_{\mathbf{1}}^T \bmu_{\mathbf{1}}}
\end{align*}
For the third term, first  denote the  constant $1 + \lambda_2  \bmu_{\mathbf{2}}^T \left(\rm I - \frac{\tilde\lambda_1 \bmu_{\mathbf{1}} \bmu_{\mathbf{1}}^T}{1 +\tilde\lambda_1 \bmu_{\mathbf{1}}^T \bmu_{\mathbf{1}}}\right) \bmu_{\mathbf{2}}$ as B: 
\begin{align*}
    &\text{Trace}\left[\frac{\left(\rm I - \frac{\tilde\lambda_1 \bmu_{\mathbf{1}} \bmu_{\mathbf{1}}^T}{1 + \tilde\lambda_1 \bmu_{\mathbf{1}}^T \bmu_{\mathbf{1}}}\right)\tilde\lambda_2 \bmu_{\mathbf{2}} \bmu_{\mathbf{2}}^T \left(\rm I - \frac{\tilde\lambda_1 \bmu_{\mathbf{1}} \bmu_{\mathbf{1}}^T}{1 + \tilde\lambda_1 \bmu_{\mathbf{1}}^T \bmu_{\mathbf{1}}}\right)}{1 + \tilde\lambda_2  \bmu_{\mathbf{2}}^T \left(\rm I - \frac{\tilde\lambda_1 \bmu_{\mathbf{1}} \bmu_{\mathbf{1}}^T}{1 + \tilde\lambda_1 \bmu_{\mathbf{1}}^T \bmu_{\mathbf{1}}}\right) \bmu_{\mathbf{2}}}  \right]  \\
    &=  \frac{1}{B}\text{Trace}\left[\left(\rm I - \frac{\tilde\lambda_1 \bmu_{\mathbf{1}} \bmu_{\mathbf{1}}^T}{1 + \tilde\lambda_1 \bmu_{\mathbf{1}}^T \bmu_{\mathbf{1}}}\right)\tilde\lambda_2 \bmu_{\mathbf{2}} \bmu_{\mathbf{2}}^T \left(\rm I - \frac{\tilde\lambda_1 \bmu_{\mathbf{1}} \bmu_{\mathbf{1}}^T}{1 + \tilde\lambda_1 \bmu_{\mathbf{1}}^T \bmu_{\mathbf{1}}}\right) \right] \\
    &= \frac{1}{B}\text{Trace}\left[\tilde\lambda_2 \bmu_{\mathbf{2}} \bmu_{\mathbf{2}}^T \left(\rm I - \frac{\tilde\lambda_1 \bmu_{\mathbf{1}} \bmu_{\mathbf{1}}^T}{1 + \tilde\lambda_1 \bmu_{\mathbf{1}}^T \bmu_{\mathbf{1}}}\right)^2\right] \text{ , by cyclic property of the Trace operator} \\
    &= \frac{1}{B}\text{Trace}\left[\tilde\lambda_2 \bmu_{\mathbf{2}} \bmu_{\mathbf{2}}^T\left(\rm I -  \frac{2  \tilde\lambda_1 \bmu_{\mathbf{1}} \bmu_{\mathbf{1}}^T}{1 + \tilde\lambda_1 \bmu_{\mathbf{1}}^T \bmu_{\mathbf{1}}} + \frac{\lambda_1^2 (\bmu_{\mathbf{1}} \bmu_{\mathbf{1}}^T)^2}{(1 + \tilde\lambda_1 \bmu_{\mathbf{1}}^T \bmu_{\mathbf{1}})^2} \right) \right]  \\
     &= \frac{1}{B} \left[\tilde\lambda_2 \bmu_{\mathbf{2}}^T\bmu_{\mathbf{2}} - \frac{2\tilde\lambda_1 \tilde\lambda_2}{1 + \tilde\lambda_1 \bmu_{\mathbf{1}}^T\bmu_{\mathbf{1}}} \text{Trace}\left[\bmu_{\mathbf{2}} \bmu_{\mathbf{2}}^T \bmu_{\mathbf{1}} \bmu_{\mathbf{1}}^T\right]  + \frac{ \tilde\lambda_2  \lambda_1^2}{(1 + \tilde\lambda_1 \bmu_{\mathbf{1}}^T \bmu_{\mathbf{1}})^2}\text{Trace}\left[(\bmu_{\mathbf{1}} \bmu_{\mathbf{1}}^T)^2\right] \right]
\end{align*}
To simplify this above expression, we note the following identities:
\begin{align*}
 \text{Trace}\left[(\bmu_{\mathbf{1}} \bmu_{\mathbf{1}}^T)^2\right] &= \left(\text{Trace}\left[\bmu_{\mathbf{1}} \bmu_{\mathbf{1}}^T \right]\right)^2\\
 &= (\bmu_{\mathbf{1}}^T \bmu_{\mathbf{1}})^2
 \end{align*}
 and similarly,
\begin{align*}
     \text{Trace}\left[(\bmu_{\mathbf{2}} \bmu_{\mathbf{2}}^T)(\bmu_{\mathbf{1}} \bmu_{\mathbf{1}}^T)\right]
     &=  \text{Trace}\left[(\bmu_{\mathbf{2}} \bmu_{\mathbf{2}}^T)\right] \text{Trace}\left[(\bmu_{\mathbf{1}} \bmu_{\mathbf{1}}^T)\right] \\
     &= \bmu_{\mathbf{2}}^T \bmu_{\mathbf{2}} \bmu_{\mathbf{1}}^T \bmu_{\mathbf{1}}
 \end{align*}
The third term then becomes: 
\begin{align*}
     \frac{1}{A} \left[\tilde\lambda_2 \bmu_{\mathbf{2}}^T\bmu_{\mathbf{2}} - \frac{\tilde\lambda_1 \tilde\lambda_2}{1 + \tilde\lambda_1 \bmu_{\mathbf{1}}^T\bmu_{\mathbf{1}}} \bmu_{\mathbf{2}}^T \bmu_{\mathbf{2}} \bmu_{\mathbf{1}}^T \bmu_{\mathbf{1}}  + \frac{ \tilde\lambda_2  \tilde\lambda_1^2}{(1 + \tilde\lambda_1 \bmu_{\mathbf{1}}^T \bmu_{\mathbf{1}})^2}(\bmu_{\mathbf{1}}^T \bmu_{\mathbf{1}})^2 \right]
\end{align*}
Finally, putting it all together, 
\begin{align*}
    &\frac{1}{n}\text{Trace}\left[(\rm I + \tilde\lambda_1  \bmu_{\mathbf{1}} \bmu_{\mathbf{1}}^T + \tilde\lambda_2  \bmu_{\mathbf{2}} \bmu_{\mathbf{2}}^T)^{-1}\right]\\
   &= \frac{1}{n}\text{Trace}\left[\rm I - \frac{\tilde\lambda_1 \bmu_{\mathbf{1}} \bmu_{\mathbf{1}}^T}{1 + \tilde\lambda_1 \bmu_{\mathbf{1}}^T \bmu_{\mathbf{1}}}  - \frac{\left(\rm I - \frac{\tilde\lambda_1 \bmu_{\mathbf{1}} \bmu_{\mathbf{1}}^T}{1 + \tilde\lambda_1 \bmu_{\mathbf{1}}^T \bmu_{\mathbf{1}}}\right)\tilde\lambda_2 \bmu_{\mathbf{2}} \bmu_{\mathbf{2}}^T \left(\rm I - \frac{\tilde\lambda_1 \bmu_{\mathbf{1}} \bmu_{\mathbf{1}}^T}{1 + \tilde\lambda_1 \bmu_{\mathbf{1}}^T \bmu_{\mathbf{1}}}\right)}{1 + \tilde\lambda_2  \bmu_{\mathbf{2}}^T \left(\rm I - \frac{\tilde\lambda_1 \bmu_{\mathbf{1}} \bmu_{\mathbf{1}}^T}{1 + \tilde\lambda_1 \bmu_{\mathbf{1}}^T \bmu_{\mathbf{1}}}\right) \bmu_{\mathbf{2}}}  \right]\\
   &= \frac{1}{n}\left[p -  \frac{\tilde\lambda_1 \bmu_{\mathbf{1}}^T \bmu_{\mathbf{1}}}{1 + \tilde\lambda_1 \bmu_{\mathbf{1}}^T \bmu_{\mathbf{1}}}  - \frac{\tilde\lambda_2 \bmu_{\mathbf{2}}^T\bmu_{\mathbf{2}} - \frac{2\tilde\lambda_1 \tilde\lambda_2}{1 + \tilde\lambda_1 \bmu_{\mathbf{1}}^T\bmu_{\mathbf{1}}} \bmu_{\mathbf{2}}^T \bmu_{\mathbf{2}} \bmu_{\mathbf{1}}^T \bmu_{\mathbf{1}}  + \frac{ \tilde\lambda_2  \tilde\lambda_1^2}{(1 + \tilde\lambda_1 \bmu_{\mathbf{1}}^T \bmu_{\mathbf{1}})^2}(\bmu_{\mathbf{1}}^T \bmu_{\mathbf{1}})^2}{1 + \tilde\lambda_2  \bmu_{\mathbf{2}}^T \left(\rm I - \frac{\tilde\lambda_1 \bmu_{\mathbf{1}} \bmu_{\mathbf{1}}^T}{1 + \tilde\lambda_1 \bmu_{\mathbf{1}}^T \bmu_{\mathbf{1}}}\right) \bmu_{\mathbf{2}}}\right] \\
     &= \frac{1}{n}\left[p -  \frac{\tilde\lambda_1 \bmu_{\mathbf{1}}^T \bmu_{\mathbf{1}}}{1 + \tilde\lambda_1 \bmu_{\mathbf{1}}^T \bmu_{\mathbf{1}}}  - \frac{\tilde\lambda_2 \bmu_{\mathbf{2}}^T\bmu_{\mathbf{2}} - \frac{2\tilde\lambda_1 \tilde\lambda_2}{1 + \tilde\lambda_1 \bmu_{\mathbf{1}}^T\bmu_{\mathbf{1}}} \bmu_{\mathbf{2}}^T \bmu_{\mathbf{2}} \bmu_{\mathbf{1}}^T \bmu_{\mathbf{1}}  + \frac{ \tilde\lambda_2  \tilde\lambda_1^2}{(1 + \tilde\lambda_1 \bmu_{\mathbf{1}}^T \bmu_{\mathbf{1}})^2}(\bmu_{\mathbf{1}}^T \bmu_{\mathbf{1}})^2}{1 + \tilde\lambda_2  \bmu_{\mathbf{2}}^T \bmu_{\mathbf{2}} - \left(\frac{\tilde\lambda_1 \tilde\lambda_2 }{1 + \tilde\lambda_1 \bmu_{\mathbf{1}}^T \bmu_{\mathbf{1}}}\right) \bmu_{\mathbf{2}}^T \bmu_{\mathbf{1}} \bmu_{\mathbf{1}}^T \bmu_{\mathbf{2}}}\right]
\end{align*}
\end{enumerate}
\end{proof}

\subsection{Proof of Theorem 2}
\begin{proof}
We consider the case in which $K = 2$ and $\|\bmu_{\mathbf{1}}\| = \|\bmu_{\mathbf{2}}\| = 1$ (that is, $\bmu_{\mathbf{1}}^T \bmu_{\mathbf{1}} = \bmu_{\mathbf{2}}^T \bmu_{\mathbf{2}} = 1$) and $\tilde\lambda_1=\tilde\lambda_2=1/2$, and simplify the quantities in Lemma \ref{lemma: twoclusters}. 
\begin{enumerate}
    \item [(i)]
 $n\mathrm{Var}\left[ \hat{Y}_{w,\rm opt}(\bx_{\star})|\bx_{\star}\right]$ is asymptotically becomes:
\begin{align*}
p - 1 + \frac{(p - 1)[1 + \tilde\lambda_1  + \tilde\lambda_2] + 1}{4(p-1) + 1 + \tilde\lambda_1 + \tilde\lambda_2} &=  p-1 + \frac{2(p-1)+ 1}{4(p-1) + 2}\\
    &= p -  \frac{1}{2} 
\end{align*}
\item [(ii)]
 $n\mathrm{Var}\left[ \hat{Y}_{M}(\bx_{\star})|\bx_{\star}\right]$ is asymptotically becomes:
\begin{align*}
     p - \frac{\tilde\lambda_1}{1 + \tilde\lambda_1} - \frac{\tilde\lambda_2  - \frac{2\tilde\lambda_1  \lambda_2}{1 + \tilde\lambda_1} + \frac{\tilde\lambda_1^2\tilde\lambda_2}{(1 +  \tilde\lambda_1)^2}}{1 - \tilde\lambda_2 - \frac{\tilde\lambda_1 \tilde\lambda_2}{1 + \tilde\lambda_1}}  &=p-\frac{1 + 2\tilde\lambda_1 \tilde\lambda_2 - \frac{3\tilde\lambda_1^2\tilde\lambda_2 + 2\tilde\lambda_1 \tilde\lambda_2}{1 + \tilde\lambda_1} + \frac{\tilde\lambda_1^2\tilde\lambda_2 + \tilde\lambda_1^3 \tilde\lambda_2}{(1 + \tilde\lambda_1)^2}}{2 + \frac{\tilde\lambda_1^2 \tilde\lambda_2 - \tilde\lambda_1\tilde\lambda_2}{1 + \tilde\lambda_1} + \tilde\lambda_1 \tilde\lambda_2}.
\end{align*}
\end{enumerate}

Now, we show that second term above is larger than $\frac{1}{2}$, and therefore that  the expected  variance of the \textit{Merged} is lower than that of the \textit{Ensemble}. Let $a= 1 + 2\tilde\lambda_1 \tilde\lambda_2 - \frac{3\tilde\lambda_1^2\tilde\lambda_2 + 2\tilde\lambda_1 \tilde\lambda_2}{1 + \tilde\lambda_1} + \frac{\tilde\lambda_1^2\tilde\lambda_2 + \tilde\lambda_1^3 \tilde\lambda_2}{(1 + \tilde\lambda_1)^2}$ (the numerator), and $b = 2 + \frac{\tilde\lambda_1^2 \tilde\lambda_2 - \tilde\lambda_1\tilde\lambda_2}{1 + \tilde\lambda_1} + \tilde\lambda_1 \tilde\lambda_2$ (the denominator). Showing that $\frac{a}{b} > \frac{1}{2}$ is equivalent to showing  that $b - 2a < 0$.
\begin{align*}
    b - 2a &= \frac{7\tilde\lambda_1^2\tilde\lambda_2  + 3\tilde\lambda_1\tilde\lambda_2}{1 + \tilde\lambda_1} - 3\tilde\lambda_1 \tilde\lambda_2  -  2\left[\frac{\tilde\lambda_1^2\tilde\lambda_2 + \tilde\lambda_1^3\tilde\lambda_2}{(1 + \tilde\lambda_1)^2}  \right] \\
    &= \frac{(7\tilde\lambda_1^2\tilde\lambda_2  + 3\tilde\lambda_1\tilde\lambda_2)(1 + \tilde\lambda_1) - 3\tilde\lambda_1 \tilde\lambda_2(1 + \tilde\lambda_1)^2 - 2[\tilde\lambda_1^2\tilde\lambda_2 + \tilde\lambda_1^3\tilde\lambda_2]}{(1 + \tilde\lambda_1)^2} \\
    &= \frac{-4\tilde\lambda_1^2\tilde\lambda_2 + 2\tilde\lambda_1^3 \lambda_2 - 6\tilde\lambda_1 \tilde\lambda_2}{(1 + \tilde\lambda_1)^2}\\
    &< 0
\end{align*}
Therefore, 
\begin{align*}
    p-\frac{1 + 2\tilde\lambda_1 \tilde\lambda_2 - \frac{3\tilde\lambda_1^2\tilde\lambda_2 + 2\tilde\lambda_1 \tilde\lambda_2}{1 + \tilde\lambda_1} + \frac{\tilde\lambda_1^2\tilde\lambda_2 + \tilde\lambda_1^3 \tilde\lambda_2}{(1 + \tilde\lambda_1)^2}}{2 + \frac{\tilde\lambda_1^2 \tilde\lambda_2 - \tilde\lambda_1\tilde\lambda_2}{1 + \tilde\lambda_1} + \tilde\lambda_1 \tilde\lambda_2} &< p -  \frac{1}{2} , 
\end{align*}
and thereby proving the desired conclusion.
\end{proof}

\subsection{Expressions for the \textit{Ensemble} and \textit{Merged} predictions
}
\label{sec:expressions}
We provide expressions for the respective predictions of the \textit{Ensemble} and \textit{Merged} on new point $\bx_{\star}$. For any vector $\bx$, we denote by $\bx_{\mathbf{S}} \in \mathbb{R}^{S}$ the covariates corresponding to only the $S$ strong features within $\bx$.

\begin{enumerate}
    \item [(i)] 

The predictions of the two cluster-level forests within the \textit{Ensemble} learner on new point $\bx_{\star}$ can be expressed as
\begin{align*}
    \hat{Y}_1(\bx_{\star}; \theta, \mathcal{D}_n) &= \sum_{i = 1}^n Y_i\mathbbm{E}_{\theta}\left[ W_{i1}(\bx_{\star}, \theta)\right] \\
    \hat{Y}_2(\bx_{\star}; \theta, \mathcal{D}_n) &= \sum_{i = 1}^n Y_i\mathbbm{E}_{\theta}\left[ W_{i2}(\bx_{\star}, \theta)\right] 
\end{align*}
where
\begin{align*}
    W_{i1}(\bx_{\star}, \theta) &= \frac{\mathbbm{1}_{\{\bx_{i} \in A_n(\bx_{\star } , \theta)\}} \mathbbm{1}_{\{\bx_{\mathbf{i} \mathbf{S}} \in [0, \frac{1}{2}]^S\}} \mathbbm{1}_{\{\bx_{\star \mathbf{S}} \in [0, \frac{1}{2}]^S\}}}{N_1(\bx_{\star \mathbf{S}} , \theta)} \mathbbm{1}_{\{\epsilon_{n1}(\bx_{\star \mathbf{S}} , \theta)\}} \\
    W_{i2}(\bx_{\star}, \theta) &= \frac{\mathbbm{1}_{\{\bx_{i} \in A_n(\bx_{\star} , \theta)\}} \mathbbm{1}_{\{\bx_{\mathbf{i} \mathbf{S}} \in [1, \frac{3}{2}]^S\}} \mathbbm{1}_{\{\bx_{\star \mathbf{S}} \in [0, \frac{1}{2}]^S\}}}{N_2(\bx_{\star \mathbf{S}} , \theta)} \mathbbm{1}_{\{\epsilon_{n2}(\bx_{\star \mathbf{S}} , \theta)\}}
\end{align*}
$\epsilon_{n1}(\bx_{\star \mathbf{S}} , \theta)$ is the event that $\sum_{i = 1}^n \mathbbm{1}_{\{\bx_{i} \in A_n(\bx_{\star} , \theta)\}} \mathbbm{1}_{\{\bx_{\mathbf{i} \mathbf{S}} \in [0, \frac{1}{2}]^S\}} \mathbbm{1}_{\{\bx_{\star \mathbf{S}} \in [0, \frac{1}{2}]^S\}} > 0$ and $\epsilon_{n2}(\bx_{\star \mathbf{S}} , \theta)$ is the event that $\sum_{i = 1}^n \mathbbm{1}_{\{\bx_{i} \in A_n(\bx_{\star } , \theta)\}} \mathbbm{1}_{\{\bx_{\mathbf{i} \mathbf{S}} \in [1, \frac{3}{2}]^S\}} \mathbbm{1}_{\{\bx_{\star \mathbf{S}} \in [0, \frac{1}{2}]^S\}} > 0$. We additionally specify that $N_1(\bx_{\star \mathbf{S}} , \theta) = \sum_{i = 1}^n (\mathbbm{1}_{\{\bx_{i} \in A_n(\bx_{\star} , \theta)\}} \mathbbm{1}_{\bx_{\mathbf{i} \mathbf{S}} \in [0, \frac{1}{2}]^S\}} \mathbbm{1}_{\bx_{\star \mathbf{S}} \in [1, \frac{3}{2}]^S\}}$, the number of total training points from $\X_1$ that fall into the same partition of the test point $\bx_{\star \mathbf{S}} $ given that $\bx_{\star \mathbf{S}} \in [1, \frac{3}{2}]^S$. A similar definition follows for $N_2(\bx_{\star \mathbf{S}} , \theta)$ with $\X_2$ and test points falling in the interval $[1, \frac{3}{2}]^S$.

We can thus represent the predictions of the overall \textit{Ensemble} as 
\begin{align*}
    \hat{Y}_E(\bx_{\star} ;\theta, \mathcal{D}_n) &= \sum_{i = 1}^n Y_i\mathbbm{E}_{\theta}\left[ W_{i1}(\bx_{\star}, \theta)\right] + Y_i\mathbbm{E}_{\theta}\left[ W_{i2}(\bx_{\star}, \theta)\right]\\
    &=  \sum_{i = 1}^n Y_i \mathbbm{E}_{\theta}\left[ W_{i}(\bx_{\star},\theta) \right]
\end{align*}

where $W_{i}(\bx_{\star},\theta) = W_{i1}(\bx_{\star}, \theta) + W_{i2}(\bx_{\star}, \theta)$. 

\item [(ii)]
The predictions of the \textit{Merged} learner on new point $\bx_{\star}$ can be expressed as
\begin{align*}
    \hat{Y}_M(\bx_{\star}; \theta, \mathcal{D}_n) &= \sum_{i = 1}^n Y_i\mathbbm{E}_{\theta}\left[ H_{i}(\bx_{\star \mathbf{S}} , \theta)\right] \\
\end{align*}
where\begin{align*}
    H_{i}(X, \theta) &= \frac{\mathbbm{1}_{\{\bx_{i} \in A_n(\bx_{\star} , \theta) \}}}{N_n(\bx_{\star \mathbf{S}} , \theta)} \mathbbm{1}_{\left\{\epsilon_n(\bx_{\star \mathbf{S}} , \theta) \right\}}
\end{align*}
and $N_n(\bx_{\star \mathbf{S}} , \theta) = \sum_{i = 1}^n \mathbbm{1}_{\{\bx_{\mathbf{i}} \in A_n(\bx_{\star} , \theta) \}}$, the number of total training samples falling into the same box as $\bx_{\star \mathbf{S}} $. Finally, $\epsilon_n(\bx_{\star \mathbf{S}} , \theta)$ is the event that $N_n(\bx_{\star \mathbf{S}} , \theta) > 0$.
\end{enumerate}

\subsection{Proof of Theorem 3}
In this section, we provide a a derivation for the upper bounds for the squared bias of the \textit{Ensemble} and \textit{Merged}.  
In order to compare the theoretical biases of individual random forest based SCLs and Merged random forest based on the whole data, we build upon the work in [\cite{klusowski2020}, version 6]. In particular, one can generally show that the leading term of the squared bias $\E[\E[\hat{Y}(\bx_{\star})|\bx_{\star}]^2 - f(\bx_{\star})]^2$ for predictor $\hat{Y}(\bx_{\star})$ is equal to
\begin{align}
    &n(n-1)\mathbbm{E}_{\bx_{\star}, \mathcal{D}_n, \beta}\left[ \E_{\theta}\left[ W_{1} \right] (\hat{Y}_1(\bx_{\star}) - f(\bx_{\star}))\mathbbm{E}_{\theta}\left[ W_{2} \right] (\hat{Y}_2(\bx_{\star}) - f(\bx_{\star}))\right]  
\end{align}
for specified weights $W_1$ and $W_2$ unique to the approach used to build the corresponding predictor. The functional forms of these weights for both the \textit{Merged} and \textit{Ensemble} are delineated above in section \ref{sec:expressions}. 
We additionally note that the $j^{th}$ coordinate of $\bx_{\star}$, denoted as $\bx_{\star}^{(j)}$ is either distributed as $\mathrm{U}\left(0, \frac{1}{2}\right)$ or $\mathrm{U}(1, \frac{3}{2})$; therefore, for a particular value of A, it has a binary expansion
\begin{align*}
   \bx_{\star}^{(j)} &= \sum_{k \geq 1} B_{kj} 2^{-k -1} + \frac{\rm A(A + 1)}{2} 
\end{align*}
where $\{ B_{kj}\}_{k = 1}^{\infty}$ are i.i.d $\text{Bernoulli}\left(\frac{1}{2}\right)$.
Throughout, we also use the following result regarding the leaf node containing test point $\bx_{\star}$.  
\begin{remark}
\label{rem: unif_leafsize}
Let $a_{nj}(\bx_{\star}, \theta)$ and $b_{nj}(\bx_{\star}, \theta)$ be the left and right endpoints of $A_{nj}(\bx_{\star}, \theta)$, the j$^{th}$ side of the box containing $\bx_{\star}$. For ease of notation, we will henceforth refer to $a_{nj}(\bx_{\star}, \theta)$ as $a_{nj}$ and $b_{nj}(\bx_{\star}, \theta)$ as $b_{nj}$. $K_{nj}(\bx_{\star}, \theta)$ represents the number of times that the j$^{th}$ coordinate is split, with the total number of splits across all coordinates set to equal $\log_2{k_n}$ for some constant $k_n > 2$. For ease of notation, we will henceforth suppress the dependencies of $a_{nj}, b_{nj}$, and $K_{nj}$ on $(\bx_{\star}, \theta)$. Recall that A = 0 if $\bx_{\star}$ falls within the interval $[0, \frac{1}{2}]^p$ and A = 1 if $\bx_{\star } \in [1, \frac{3}{2}]^p$.
We then observe that each endpoint of $A_{nj}(\bx_{\star}, \theta)$ is a randomly stopped binary expansion of $\bx_{\star}^{(j)}$: 
\begin{align*}
    a_{nj} &\overset{D}{=} \sum_{k = 1}^{K_{nj}} B_{kj} 2^{-k-1}  + \frac{\rm A(A + 1)}{2} \\
    b_{nj} &\overset{D}{=} \sum_{k = 1}^{K_{nj}} B_{kj} 2^{-k -1} + 2^{-{ K_{nj} - 1}} + \frac{\rm A(A + 1)}{2}
\end{align*}
 No matter the value of A, the length of the j$^{th}$ side of the box is given by:
\begin{align*}
   {\lambda(A_{nj}) = b_{nj} - a_{nj} = 2^{-K_{nj} - 1}}
\end{align*}
Therefore, the measure of the box $A_n$ is equal to
\begin{align*}
    \lambda(A_n) &= \prod_{j \in [d]} \lambda(A_{nj}) \\
    &= \prod_{j \in [d]} 2^{-K_{nj} - 1} \\
    &= 2^{-\lceil \log_2k_n \rceil - 1}
\end{align*}
since by construction, $\sum_j K_{nj} = \lceil \log_2k_n \rceil$.
\end{remark}

In order to compare the theoretical biases of individual random forest based SCL based on cluster $t$ and Merged random forest based on the whole data, we build upon the work in [\cite{klusowski2020}, version 6]. In particular, following the arguments [\cite{klusowski2020}, version 6], one can show that the leading term of the squared bias $\E[\E[\hat{Y}_E(\bx_{\star})|\bx_{\star}]^2 - f(\bx_{\star})]^2$ is equal to
\begin{align}
    &n(n-1)\mathbbm{E}_{\bx_{\star}, \mathcal{D}_n, \beta}\left[ \E_{\theta}\left[ W_{1} \right] (\hat{Y}_1(\bx_{\star}) - f(\bx_{\star}))\mathbbm{E}_{\theta}\left[ W_{2} \right] (\hat{Y}_2(\bx_{\star}) - f(\bx_{\star}))\right]  
\end{align}
for specified weights $W_1$ and $W_2$ unique to the approach used to build the corresponding predictor. The functional forms of these weights for both the \textit{Merged} and \textit{Ensemble} are delineated in Appendix section \ref{sec:expressions}. 
The notation $\mathbbm{1}_{\{\mathrm{U}\left(0, \frac{1}{2}\right) \in A_{nj}\}}$ will indicate the event that a random variable distributed as a $\mathrm{U}\left(0, \frac{1}{2}\right)$ falls within the interval $A_{nj}$. 
Given this result, we proceed 
\begin{proof}
\begin{enumerate}
\item[(i)]
The leading term of the squared bias for the \textit{Ensemble} can be represented as
\begin{align*}
n(n-1)\mathbbm{E}_{\bx_{\star}, \mathcal{D}_n, \beta}\left[ \mathbbm{E}_{\theta}\left[ W_{1}\right] (f(\bx_1) - f(\bx_{\star}))\mathbbm{E}_{\theta}\left[ W_{2} \right] (f(\bx_2) - f(\bx_{\star}))\right] 
\end{align*}
where $W_{i} = W_{i1}(\bx_{\star}, \theta) + W_{i2}(\bx_{\star}, \theta)$ for $i = 1, 2$,  as defined in Section \ref{sec:expressions} (i). 
We first take expectations w.r.t $\beta$ of the expression inside the outer expectation. We use the fact that for a quadratic form, 
\begin{align*}
    E[\epsilon^T \Lambda \epsilon] = \text{Trace}(\Lambda \Sigma) + \mu^T \Lambda \mu
\end{align*}
where $\mu = E[\epsilon]$ and $\Sigma = Var[\epsilon]$
\begin{align*}
    &\mathbbm{E}_{\beta}\left[ \mathbbm{E}_{\theta}\left[ W_{1} \right] (f(\bx_{\mathbf{1} \mathbf{S}}) - f(\bx_{\star \mathbf{S}} ))\mathbbm{E}_{\theta}\left[ W_{2} \right] (f(\bx_{\mathbf{2} \mathbf{S}}) - f(\bx_{\star \mathbf{S}} ))\right] \\
    = &\mathbbm{E}_{\theta}\left[ W_{1} \right] \mathbbm{E}_{\theta}\left[ W_{2} \right] \mathbbm{E}_{\beta}\left[\beta^T(\bx_{\mathbf{1} \mathbf{S}} - \bx_{\star \mathbf{S}} )(\bx_{\mathbf{2} \mathbf{S}}-\bx_{\star \mathbf{S}} )^T\beta \right] \\
    = &\mathbbm{E}_{\theta}\left[ W_{1} \right] \mathbbm{E}_{\theta}\left[ W_{2} \right] \text{Trace}[(\bx_{\mathbf{1} \mathbf{S}} - \bx_{\star \mathbf{S}} ) (\bx_{\mathbf{2} \mathbf{S}} - \bx_{\star \mathbf{S}} )^T]
\end{align*}

Therefore, 
\begin{align*}
    &\mathbbm{E}_{\bx_{\star} , \mathcal{D}_n, \beta}\left[ \mathbbm{E}_{\theta}\left[ W_{1} \right] (f(\bx_{\mathbf{1} \mathbf{S}}) - f(\bx_{\star \mathbf{S}} ))\mathbbm{E}_{\theta}\left[W_{2} \right] (f(\bx_{\mathbf{2} \mathbf{S}}) - f(\bx_{\star \mathbf{S}} ))\right] \\
    &= \mathbbm{E}_{\bx_{\star}, \mathcal{D}_n}\left[ \mathbbm{E}_{\theta}\left[ W_{1} \right] \mathbbm{E}_{\theta}\left[ W_{2} \right] \text{Trace}[(\bx_{\mathbf{1} \mathbf{S}} - \bx_{\star \mathbf{S}} ) (\bx_{\mathbf{2} \mathbf{S}} - \bx_{\star \mathbf{S}} )^T] \right] \\
    &= \mathbbm{E}_{\bx_{\star}, \mathcal{D}_n, \theta, \theta^{'}}\left[ W_{1} W_{2}^{'} \text{Trace}[(\bx_{\mathbf{1} \mathbf{S}} - \bx_{\star \mathbf{S}} ) (\bx_{\mathbf{2} \mathbf{S}} - \bx_{\star \mathbf{S}} )^T] \right]
\end{align*}
where $\theta^{'}$ is an independent copy of $\theta$. Next, we calculate the product $W_{i1} W_{i2}^{'}$:
\begin{align*}
    W_{1} W_{2}^{'} &= \mathbbm{1}_{\{\bx_{\mathbf{1}} \in A_n\}} \mathbbm{1}_{\{\bx_{\mathbf{2}} \in A_n^{'}\}} \\ 
    &\times \left[\frac{\mathbbm{1}_{\{\bx_{\mathbf{1} \mathbf{S}} \in [0, \frac{1}{2}]^S\}} \mathbbm{1}_{\{\bx_{\star \mathbf{S}} \in [0, \frac{1}{2}]^S\}}}{N_1} \mathbbm{1}_{\{\epsilon_{n1}\}} + \frac{\mathbbm{1}_{\{\bx_{\mathbf{1} \mathbf{S}} \in [1, \frac{3}{2}]^S\}} \mathbbm{1}_{\{\bx_{\star \mathbf{S}} \in [1, \frac{3}{2}]^S\}}}{N_2}
    \mathbbm{1}_{\{\epsilon_{n2}\}} \right] \\
    &\times \left[\frac{\mathbbm{1}_{\{\bx_{\mathbf{2} \mathbf{S}} \in [0, \frac{1}{2}]^S\}} \mathbbm{1}_{\{\bx_{\star \mathbf{S}} \in [0, \frac{1}{2}]^S\}}}{N_1^{'}} \mathbbm{1}_{\{\epsilon_{n1}^{'}\}} + \frac{\mathbbm{1}_{\{\bx_{\mathbf{2} \mathbf{S}} \in [1, \frac{3}{2}]^S\}} \mathbbm{1}_{\{\bx_{\star \mathbf{S}} \in [1, \frac{3}{2}]^S\}}}{N_2^{'}}
    \mathbbm{1}_{\{\epsilon_{n2}^{'}\}} \right] \\
    &= \mathbbm{1}_{\{\bx_{\mathbf{1}} \in A_n\}} \mathbbm{1}_{\{\bx_{\mathbf{2}} \in A_n^{'}\}} \\ 
    &\times\Bigg[\frac{\mathbbm{1}_{\{\bx_{\mathbf{1} \mathbf{S}} \in [0, \frac{1}{2}]^S\}}\mathbbm{1}_{\{\bx_{\mathbf{2} \mathbf{S}} \in [0, \frac{1}{2}]^S\}}[\mathbbm{1}_{\{\bx_{\star \mathbf{S}} \in [0, \frac{1}{2}]^S\}}]^2 \mathbbm{1}_{\{\epsilon_{n1}\}} \mathbbm{1}_{\{\epsilon_{n1}^{'}\}}}{N_1 N_1^{'}}  \\
    &+ \frac{\mathbbm{1}_{\{\bx_{\mathbf{1} \mathbf{S}} \in [1, \frac{3}{2}]^S\}}\mathbbm{1}_{\{\bx_{\mathbf{2} \mathbf{S}} \in [1, \frac{3}{2}]^S\}}[\mathbbm{1}_{\{\bx_{\star \mathbf{S}} \in [1, \frac{3}{2}]^S\}}]^2 \mathbbm{1}_{\{\epsilon_{n2}\}} \mathbbm{1}_{\{\epsilon_{n2}^{'}\}}}{N_2 N_2^{'}} \Bigg]
\end{align*}
The cross terms of the product are equal to 0, since $\mathbbm{1}_{\{\bx_{\star \mathbf{S}} \in [0, \frac{1}{2}]^S\}} \times \mathbbm{1}_{\{\bx_{\star \mathbf{S}} \in [1, \frac{3}{2}]^S\}} = 0$. Define 
\begin{align*}
    T_1 &= \sum_{i \geq 3} \mathbbm{1}_{\{\bx_i  \in A_n\}}\mathbbm{1}_{\{\bx_{\mathbf{i} \mathbf{S}}  \in [0, \frac{1}{2}]^S\}} \mathbbm{1}_{\{\bx_{\star \mathbf{S}} \in [0, \frac{1}{2}]^S\}} \\
    T_2 &= \sum_{i \geq 3} \mathbbm{1}_{\{\bx_i  \in A_n\}}\mathbbm{1}_{\{\bx_{\mathbf{i} \mathbf{S}}  \in [1, \frac{3}{2}]^S\}} \mathbbm{1}_{\{\bx_{\star \mathbf{S}} \in [1, \frac{3}{2}]^S\}}
\end{align*}
and let $T_1^{'}$ and $T_2^{'}$ be the equivalent expressions based on $\theta^{'}$. We then note that 
\begin{align*}
    W_{1} W_{2}^{'} &\leq \mathbbm{1}_{\{\bx_{\mathbf{1}} \in A_n\}} \mathbbm{1}_{\{\bx_{\mathbf{2}} \in A_n^{'}\}} \\ 
    &\times\Bigg[\frac{\mathbbm{1}_{\{\bx_{\mathbf{1} \mathbf{S}} \in [0, \frac{1}{2}]^S\}}\mathbbm{1}_{\{\bx_{\mathbf{2} \mathbf{S}} \in [0, \frac{1}{2}]^S\}}[\mathbbm{1}_{\{\bx_{\star \mathbf{S}} \in [0, \frac{1}{2}]^S\}}]^2}{(1+T_1)(1+T_1^{'})}  \\
    &+ \frac{\mathbbm{1}_{\{\bx_{\mathbf{1} \mathbf{S}} \in [1, \frac{3}{2}]^S\}}\mathbbm{1}_{\{\bx_{\mathbf{2} \mathbf{S}} \in [1, \frac{3}{2}]^S\}}[\mathbbm{1}_{\{\bx_{\star \mathbf{S}} \in [1, \frac{3}{2}]^S\}}]^2 }{(1+T_2)(1+T_2^{'})} \Bigg]
\end{align*}

By the Cauchy-Schwarz inequality,
\begin{align*}
    \mathbbm{E}\left[\frac{1}{(1 + T_1)(1 + T_1^{'})} \bigg|\bx_{\star \mathbf{S}} , \theta, \theta^{'} \right] &\leq \sqrt{\mathbbm{E}\left[\left(\frac{1}{1 + T_1}\right)^2  \bigg|\bx_{\star \mathbf{S}} , \theta, \theta^{'}\right]} \sqrt{\mathbbm{E}\left[\left(\frac{1}{1 + T_1^{'}}\right)^2  \bigg|\bx_{\star \mathbf{S}} , \theta, \theta^{'}\right]} \\
    &\leq \left(\frac{2^{\lceil \log_2k_n \rceil + 2 }}{\sqrt{n(n-1)}}\right)^2 \\
    &= \frac{4^{\lceil \log_2k_n \rceil + 2}}{n(n-1)}
\end{align*}
since if a variable $Z \sim \text{Binomial}(m, p)$, $\mathbbm{E}\left[\left(\frac{1}{1+Z} \right)^2\right] \leq \frac{1}{(m + 1)(m + 2) p^2}$.
 The binomial probability $p$ in this case is equal to
\begin{align*}
    &\mathbbm{E}\left[\mathbbm{1}_{\{\bx_i  \in A_n\}}\mathbbm{1}_{\{\bx_{\mathbf{i} \mathbf{S}}  \in [0, \frac{1}{2}]^S\}} \mathbbm{1}_{\{\bx_{\star \mathbf{S}} \in [0, \frac{1}{2}]^S\}} \right] \\
    &= \mathbbm{E}\left[\mathbbm{1}_{\{\bx_i  \in A_n\}} \bigg| \mathbbm{1}_{\{\bx_{\mathbf{i} \mathbf{S}}  \in [0, \frac{1}{2}]^S\}} \mathbbm{1}_{\{\bx_{\star \mathbf{S}} \in [0, \frac{1}{2}]^S\}} \right] \times \mathbbm{E}\left[\mathbbm{1}_{\{\bx_{\mathbf{i} \mathbf{S}}  \in [0, \frac{1}{2}]^S\}} \right] \times \mathbbm{E}\left[\mathbbm{1}_{\{\bx_{\star \mathbf{S}} \in [0, \frac{1}{2}]^S\}} \right] \\
    &= \frac{1}{4} \mathbbm{E}\left[\mathbbm{1}_{\{\bx_i  \in A_n\}} \bigg| \mathbbm{1}_{\{\bx_{\mathbf{i} \mathbf{S}}  \in [0, \frac{1}{2}]^S\}} \mathbbm{1}_{\{\bx_{\star \mathbf{S}} \in [0, \frac{1}{2}]^S\}} \right] \\
    &= \frac{1}{4} \prod_{j = 1}^S 2(b_{nj} - a_{nj}) \\
    &= \frac{1}{4} \prod_{j = 1}^S 2 \times 2^{-K_{nj} - 1} \\
    &= 2^{-\lceil \log_2k_n \rceil - 2}
\end{align*}

It is simple to see that by a symmetric argument,  $\mathbbm{E}\left[\frac{1}{(1 + T_2)(1 + T_2^{'})} \bigg| \bx_{\star \mathbf{S}} , \theta, \theta^{'} \right]$ has the same upper bound.
Then,
\begin{alignat*}{2}
    &\mathbbm{E}_{\bx_{\star}, \mathcal{D}_n, \theta, \theta^{'}}\left[ W_{1} W_{2}^{'} \text{Trace}[(\bx_{\mathbf{1} \mathbf{S}} - \bx_{\star \mathbf{S}} ) (\bx_{\mathbf{2} \mathbf{S}} - \bx_{\star \mathbf{S}} )^T] \right] \\
    &= 2 \mathbbm{E}_{\bx_{\star} , \mathcal{D}_n}\Bigg[ \mathbbm{E}\left[\frac{1}{(1 + T_1)(1 + T_1^{'})} | \bx_{\star \mathbf{S}} , \theta, \theta^{'} \right] \\
    &\times \mathbbm{E}_{\bx_1, \bx_{\mathbf{2} }}\bigg[\mathbbm{1}_{\{\bx_{\mathbf{1}} \in A_n\}} \mathbbm{1}_{\{\bx_{\mathbf{2}} \in A_n^{'}\}}\mathbbm{1}_{\{\bx_{\mathbf{1} \mathbf{S}} \in [0, \frac{1}{2}]^S\}} \mathbbm{1}_{\{\bx_{\mathbf{2} \mathbf{S}} \in [0, \frac{1}{2}]^S\}} \left[\mathbbm{1}_{\{\bx_{\star \mathbf{S}} \in [0, \frac{1}{2}]^S\}}\right]^2 \text{Trace}\left[(\bx_{\mathbf{1} \mathbf{S}} - \bx_{\star \mathbf{S}} ) (\bx_{\mathbf{2} \mathbf{S}} - \bx_{\star \mathbf{S}} )^T\right] \bigg] \Bigg]\\
    &= \frac{4^{\lceil \log_2k_n \rceil + 5/2}}{n(n-1)} \\
    &\times \mathbbm{E}\Bigg[ \mathbbm{E}\bigg[ \mathbbm{1}_{\{\bx_{\mathbf{1}} \in A_n\}} \mathbbm{1}_{\{\bx_{\mathbf{2}} \in A_n^{'}\}}\mathbbm{1}_{\{\bx_{\mathbf{1} \mathbf{S}} \in [0, \frac{1}{2}]^S\}} \mathbbm{1}_{\{\bx_{\mathbf{2} \mathbf{S}} \in [0, \frac{1}{2}]^S\}} \left[\mathbbm{1}_{\{\bx_{\star \mathbf{S}} \in [0, \frac{1}{2}]^S\}}\right]^2 \text{Trace}\left[(\bx_{\mathbf{1} \mathbf{S}} - \bx_{\star \mathbf{S}} ) (\bx_{\mathbf{2} \mathbf{S}} - \bx_{\star \mathbf{S}} )^T\right] |\bx_{\star \mathbf{S}}  \bigg] \Bigg]\\
\end{alignat*}

 We now look the inner expectation:
\begin{align*}
    &\mathbbm{E}\bigg[ \mathbbm{1}_{\{\bx_{\mathbf{1}} \in A_n\}} \mathbbm{1}_{\{\bx_{\mathbf{2} }\in A_n^{'}\}}\mathbbm{1}_{\{\bx_{\mathbf{1} \mathbf{S}} \in [0, \frac{1}{2}]^S\}} \mathbbm{1}_{\{\bx_{\mathbf{2} \mathbf{S}}\in [0, \frac{1}{2}]^S\}} \left[\mathbbm{1}_{\{\bx_{\star \mathbf{S}} \in [0, \frac{1}{2}]^S\}}\right]^2 \text{Trace}[(\bx_{\mathbf{1} \mathbf{S}} - \bx_{\star \mathbf{S}} ) (\bx_{\mathbf{2} \mathbf{S}} - \bx_{\star \mathbf{S}} )^T] |\bx_{\star \mathbf{S}}  \bigg]  \\
    &= \mathbbm{E}\left[ \mathbbm{1}_{\{\bx_{\mathbf{1}} \in A_n\}} \mathbbm{1}_{\{\bx_{\mathbf{2} }\in A_n^{'}\}}\mathbbm{1}_{\{\bx_{\mathbf{1} \mathbf{S}} \in [0, \frac{1}{2}]^S\}} \mathbbm{1}_{\{\bx_{\mathbf{2} \mathbf{S}}\in [0, \frac{1}{2}]^S\}} \left[\mathbbm{1}_{\{\bx_{\star \mathbf{S}} \in [0, \frac{1}{2}]^S\}}\right]^2 \sum_{j}(\bx_{\mathbf{1} \mathbf{S}}^{(j)} - \bx_{\star \mathbf{S}} ^{(j)}) (\bx_{\mathbf{2} \mathbf{S}}^{(j)} - \bx_{\star \mathbf{S}} ^{(j)}) \big| \bx_{\star \mathbf{S}}  \right] \\
    &=  \sum_{j = 1}^S \mathbbm{E}\left[  \mathbbm{1}_{\{\bx_{\mathbf{1}} \in A_n\}} \mathbbm{1}_{\{\bx_{\mathbf{2} }\in A_n^{'}\}}\mathbbm{1}_{\{\bx_{\mathbf{1} \mathbf{S}} \in [0, \frac{1}{2}]^S\}} \mathbbm{1}_{\{\bx_{\mathbf{2} \mathbf{S}}\in [0, \frac{1}{2}]^S\}} \left[\mathbbm{1}_{\{\bx_{\star \mathbf{S}} \in [0, \frac{1}{2}]^S\}}\right]^2 (\bx_{\mathbf{1} \mathbf{S}}^{(j)} - \bx_{\star \mathbf{S}} ^{(j)}) (\bx_{\mathbf{2} \mathbf{S}}^{(j)} - \bx_{\star \mathbf{S}} ^{(j)})  \big|
    \bx_{\star \mathbf{S}}  \right] \\
    &= \sum_{j = 1}^S \Bigg( \mathbbm{E}\left[ \mathbbm{1}_{\{\bx_{\mathbf{1}} \in A_n \}} \mathbbm{1}_{\{\bx_{\mathbf{1} \mathbf{S}} \in [0, \frac{1}{2}]^S\}} \mathbbm{1}_{\{\bx_{\star \mathbf{S}} \in [0, \frac{1}{2}]^S\}} (\bx_{\mathbf{1} \mathbf{S}}^{(j)} - \bx_{\star \mathbf{S}} ^{(j)}) \big| \bx_{\star \mathbf{S}}  \right] \\
    &\times \mathbbm{E}\left[ \mathbbm{1}_{\{\bx_2\in A_n \}}\mathbbm{1}_{\{\bx_{\mathbf{2} \mathbf{S}}\in [0, \frac{1}{2}]^S\}} \mathbbm{1}_{\{\bx_{\star \mathbf{S}} \in [0, \frac{1}{2}]^S\}} (\bx_{\mathbf{2} \mathbf{S}}^{(j)} - \bx_{\star \mathbf{S}} ^{(j)})\big| \bx_{\star \mathbf{S}}  \right] \Bigg)\\
    &= \sum_{j= 1}^S \left(\mathbbm{E}\left[ \mathbbm{1}_{\{\bx_{\mathbf{1}} \in A_n \}} \mathbbm{1}_{\{\bx_{\mathbf{1} \mathbf{S}} \in [0, \frac{1}{2}]^S\}} \mathbbm{1}_{\{\bx_{\star \mathbf{S}} \in [0, \frac{1}{2}]^S\}} |\bx_{\mathbf{1} \mathbf{S}}^{(j)} - \bx_{\star \mathbf{S}} ^{(j)}| \big| \bx_{\star \mathbf{S}}  \right] \right)^2
\end{align*}
Consider the expression within the square:
\begin{align*}
    &\mathbbm{E}\left[ \mathbbm{1}_{\{\bx_{\mathbf{1}} \in A_n \}} \mathbbm{1}_{\{\bx_{\mathbf{1} \mathbf{S}} \in [0, \frac{1}{2}]^S\}} \mathbbm{1}_{\{\bx_{\star \mathbf{S}} \in [0, \frac{1}{2}]^S\}} |\bx_{\mathbf{1} \mathbf{S}}^{(j)} - \bx_{\star \mathbf{S}} ^{(j)}| \big| \bx_{\star \mathbf{S}}  \right] \\
    &= \mathbbm{E}\left[  |\bx_{\mathbf{1} \mathbf{S}}^{(j)} - \bx_{\star \mathbf{S}} ^{(j)}| \bigg| \mathbbm{1}_{\{\bx_{\mathbf{1} \mathbf{S}}^{(j)} \in A_{nj} \}},\mathbbm{1}_{\{\bx_{\mathbf{1} \mathbf{S}} \in [0, \frac{1}{2}]^S\}},  \mathbbm{1}_{\{\bx_{\star \mathbf{S}} \in [0, \frac{1}{2}]^S\}}, \bx_{\star \mathbf{S}}  \right] \\
    &\times \mathbbm{E}\left[ \mathbbm{1}_{\{\bx_{\mathbf{1}} \in A_n \}} |  \mathbbm{1}_{\{\bx_{\mathbf{1} \mathbf{S}} \in [0, \frac{1}{2}]^S\}},  \mathbbm{1}_{\{\bx_{\star \mathbf{S}} \in [0, \frac{1}{2}]^S\}}, \bx_{\star \mathbf{S}}  \right] \times \mathbbm{E}\left[\mathbbm{1}_{\{\bx_{\mathbf{1} \mathbf{S}} \in [0, \frac{1}{2}]^S\}} \right] \times \mathbbm{E}\left[\mathbbm{1}_{\{\bx_{\star \mathbf{S}} \in [0, \frac{1}{2}]^S\}} \right]\\
    &\leq {\frac{1}{2} \lambda(A_{nj}) \times 2\lambda(A_n)} \times \frac{1}{4} \\
    &= \frac{1}{4} \lambda(A_{nj}) \times \lambda(A_n)
\end{align*}
Note that we calculate an upper bound for $ \mathbbm{E}\left[  |\bx_{\mathbf{1} \mathbf{S}}^{(j)} - \bx_{\star \mathbf{S}} ^{(j)}| \bigg| \mathbbm{1}_{\{\bx_{\mathbf{1} \mathbf{S}}^{(j)} \in A_{nj} \}},\mathbbm{1}_{\{\bx_{\mathbf{1} \mathbf{S}} \in [0, \frac{1}{2}]^S\}},  \mathbbm{1}_{\{\bx_{\star \mathbf{S}} \in [0, \frac{1}{2}]^S\}}, \bx_{\star \mathbf{S}}  \right]$ by noting that the most extreme value of $\bx_{\star \mathbf{S}} ^{(j)}$ is if it is one of the endpoints of the interval. In that case, the difference $|\bx_{\mathbf{1} \mathbf{S}}^{(j)} - \bx_{\star \mathbf{S}} ^{(j)}|$ can take on any number between 0 and $\lambda(A_{nj})$ with equal probability, since $\bx_{\mathbf{1} \mathbf{S}}^{(j)}| \mathbbm{1}_{\{\bx_{\mathbf{1} \mathbf{S}}^{(j)} \in A_{nj} \}}$ is uniformly distributed. Therefore, the expected value of the difference in this case is $\frac{1}{2}\lambda(A_{nj})$. Any other value of $\bx_{\star \mathbf{S}} ^{(j)}$ within the interval will produce a smaller expected value of the difference. 
Therefore, 
\begin{align*}
    &\sum_j \left(\mathbbm{E}\left[ \mathbbm{1}_{\{\bx_{\mathbf{1}} \in A_n \}} \mathbbm{1}_{\{\bx_{\mathbf{1} \mathbf{S}} \in [0, \frac{1}{2}]^S\}} \mathbbm{1}_{\{\bx_{\star \mathbf{S}} \in [0, \frac{1}{2}]^S\}} \left|\bx_{\mathbf{1} \mathbf{S}}^{(j)} - \bx_{\star \mathbf{S}} ^{(j)}\right| \big| \bx_{\star \mathbf{S}}  \right] \right)^2 \\
    &= \sum_j \left[\frac{1}{4} \lambda(A_{nj}) \times \lambda(A_n)\right]^2\\
    &= 4^{-2}\lambda^2(A_n) \sum_j \lambda^2(A_{nj}) \\
    &= 4^{-\lceil \log_2k_n \rceil - 4} \sum_j \left[2^{-K_{nj}}\right]^2 \\
\end{align*}

We now take it's expectation w.r.t. $\theta$, yielding
\begin{align*}
   \mathbbm{E}\left[ 4^{-\lceil \log_2k_n \rceil - 4} \sum_j [2^{-K_{nj}}]^2  \right] &= 4^{-\lceil \log_2k_n \rceil - 4} \sum_j \mathbbm{E}\left[[2^{-K_{nj}}]^2  \right]
\end{align*}

Recall that for any $j$, $K_{nj} \sim \text{Binomial}(p_{nj}, \lceil \log_2k_n \rceil)$. We do a transformation to obtain the distribution of $y = \left[2^{-K_{nj}}\right]^2$. Note that $K_{nj} = -\log_2y$. Therefore, 
\begin{align*}
    f_Y(y) &= f_{k_{nj}}(-\log_2y) \\
    &= \binom{\lceil \log_2k_n \rceil}{-\log_2y} p_{nj}^{-\log_2y}(1-p_{nj})^{\lceil \log_2k_n \rceil + \log_2y}
\end{align*}
and 
\begin{align*}
    E[Y^2] &= \sum_{y \in \{1, \frac{1}{2}, ..., 2^{-\lceil \log_2k_n \rceil} \}} y^2 \binom{\lceil \log_2k_n \rceil}{-\log_2y} p_{nj}^{-\log_2y}(1-p_{nj})^{\lceil \log_2k_n \rceil + \log_2y} \\
    &= \sum_{k = 0}^{\log_2k_n} \binom{\lceil \log_2k_n \rceil}{k} 2^{-2k} (p_{nj})^k (1-p_{nj})^{\lceil \log_2k_n \rceil -k} \\
    &= 2^{-\lceil \log_2k_n \rceil}\sum_{k = 0}^{\log_2k_n} \binom{\lceil \log_2k_n \rceil}{k}  \left(\frac{p_{nj}}{2}\right)^k (2-2p_{nj})^{\lceil \log_2k_n \rceil -k} \\
    &= 2^{-\lceil \log_2k_n \rceil}\left(\frac{p_{nj}}{2} + 2-2p_{nj} \right)^{\lceil \log_2k_n \rceil} \\
    &= \left(1 - \frac{3p_{nj}}{4} \right)^{\lceil \log_2k_n \rceil}\\
    &\leq {k_n^{\log_2\left(1 - 3p_{n}/4 \right)}}
\end{align*}
Empirically, the last bound is tight. 
Now, using this bound, 
\begin{align*}
    4^{-\lceil \log_2k_n \rceil - 4} \sum_j \mathbbm{E}\left[[2^{-K_{nj}}]^2 \big|\bx_{\star \mathbf{S}}  \right] & \leq 4^{-\lceil \log_2k_n \rceil - 4} \sum_j k_n^{log_2(1 - 3p_n/4)} \\
    &= 4^{-\lceil \log_2k_n \rceil - 4} \times \text{S} \times k_n^{log_2(1 - 3p_n/4)}
\end{align*}

Putting it all together, 
\begin{align*}
    &\frac{4^{\lceil \log_2k_n \rceil + 5/2}}{n(n-1)}  \mathbbm{E}\left[\sum_j \left(\mathbbm{E}\left[ \mathbbm{1}_{\{\bx_{\mathbf{1}} \in A_n \}} \mathbbm{1}_{\{\bx_{\mathbf{1} \mathbf{S}} \in [0, \frac{1}{2}]^S\}} \mathbbm{1}_{\{\bx_{\star \mathbf{S}} \in [0, \frac{1}{2}]^S\}} |\bx_{\mathbf{1} \mathbf{S}}^{(j)} - \bx_{\star \mathbf{S}} ^{(j)}| \big|\bx_{\star \mathbf{S}} \right] \right)^2  \right] \\
    &\leq \frac{4^{\lceil \log_2k_n \rceil + 5/2}}{n(n-1)} 4^{-\lceil \log_2k_n \rceil - 4} \text{S} k_n^{log_2(1 - 3p_n/4)} \\
    &= \frac{\text{S} k_n^{log_2(1 - 3p_n/4)}}{8n(n-1)}
\end{align*}

 \item[(ii)]
 The leading term of the squared bias for the \textit{Merged} is equal to
\begin{align*}
n(n-1)E_{\bx_{\star}, \mathcal{D}_n, \beta}\left[ \E_{\theta}\left[ H_{1} \right] (f(\bx_{\mathbf{1} \mathbf{S}}) - f(\bx_{\star \mathbf{S}} ))\E_{\theta}\left[ H_{2} \right] (f(\bx_{\mathbf{2} \mathbf{S}}) - f(\bx_{\star \mathbf{S}} ))\right]
\end{align*}
where $H_1$ and $H_2$ are defined in Section \ref{sec:expressions} (ii). As in the previous section, we begin by calculating
\begin{align*}
    &E_{\bx_1^{(j)}}\left[\mathbbm{1}_{\{\bx_{\mathbf{1} \mathbf{S}}^{(j)} \in A_{nj}\}} \right] \\
    &= E_{\bx_1^{(j)}}\left[\mathbbm{1}_{\{\mathrm{U}\left(0, \frac{1}{2}\right) \in A_{nj}\}} \mathbbm{1}\{i \in \mathbb{S}_1\} + \mathbbm{1}_{\{\mathrm{U}(1, \frac{3}{2}) \in A_{nj}\}}\mathbbm{1}\{i \in \mathbb{S}_2\} \right] \\
    &= E_i\bigg[E_{\bx_1^{(j)}|i}\left[\mathbbm{1}_{\{\mathrm{U}\left(0, \frac{1}{2}\right) \in A_{nj}\}}\bigg|i \in \mathbb{S}_1 \right]\mathbbm{1}\{i \in \mathbb{S}_1\} +
    E_{\bx_1^{(j)}|i}\left[\mathbbm{1}_{\{\mathrm{U}(1, \frac{3}{2}) \in A_{nj}\}}\bigg|i \in \mathbb{S}_2 \right]\mathbbm{1}\{i \in \mathbb{S}_2\}\bigg] \\
    &= E_{\bx_1^{(j)}|i}\left[\mathbbm{1}_{\{\mathrm{U}\left(0, \frac{1}{2}\right) \in A_{nj}\}}\bigg|i \in \mathbb{S}_1 \right] \\
    &= P\left(\mathrm{U}\left(0, \frac{1}{2}\right) \in A_{nj}\right)
\end{align*}
To find this last quantity, note that for test point $ \bx_{\star} \sim [\mathrm{U}\left(0, \frac{1}{2}\right)]^S$ or $[\mathrm{U}\left(1, \frac{3}{2}\right)]^S$, there are two possible scenarios as to the values of $a_{nj}$ and $b_{nj}$:
\begin{enumerate}
    \item $a_{nj} < \frac{1}{2}, b_{nj} \leq \frac{1}{2}$
    \item $a_{nj} \geq 1, b_{nj} > 1$
\end{enumerate}
Note that $P(a_{nj} \geq 1) = P(B_{1j}(\bx_{\star}) = 1) = P(\text{Bernoulli}\left(\frac{1}{2}\right) = 1) = \frac{1}{2}$. Then, taking the second scenario as an example:
\begin{align*}
    P(a_{nj} \geq 1, b_{nj} > 1) &= P(b_{nj} > 1 | a_{nj} \geq 1)P(a_{nj} \geq 1) \\
    &= 1 \times \frac{1}{2} = \frac{1}{2}
\end{align*}
We see then that the probability of either of the scenarios above is equal to $\frac{1}{2}$. We can then calculate that:
\begin{align*}
    P\left(\mathrm{U}\left(0, \frac{1}{2}\right) \in A_{nj}\right) &= P\left(\mathrm{U}\left(0, \frac{1}{2}\right) \in A_{nj}\right|a_{nj} < \frac{1}{2}, b_{nj} \leq \frac{1}{2})P(a_{nj} < \frac{1}{2}, b_{nj} \leq \frac{1}{2}) \\
    &+ P\left(\mathrm{U}\left(0, \frac{1}{2}\right) \in A_{nj}\right|a_{nj} \geq 1, b_{nj} > 1)P(a_{nj} \geq 1, b_{nj} > 1) \\
    &= \frac{1}{2}\left[\frac{b_{nj} - a_{nj}}{\frac{1}{2}} + 0 \right] \\
    &= b_{nj} - a_{nj}
\end{align*}

We now define $U = \sum_{i \geq 3} \mathbbm{1}_{\{\bx_{\mathbf{i}} \in A_n\}} \sim \text{Binomial}(n-2, p)$ where $p = E_{\bx_1} \left[ \mathbbm{1}_{\{\bx_{\mathbf{1} \mathbf{S}} \in A_n \}} \right] = 2^{-\lceil \log_2k_n \rceil - 1}$. Next, note that 

\begin{align*}
    H_{1}H^{'}_{2} \leq \frac{\mathbbm{1}_{\{\bx_{\mathbf{1}} \in A_n\}} \mathbbm{1}_{\{\bx_{\mathbf{2}} \in A_n^{'}\}}}{(1 + U)(1 + U^{'})}
\end{align*}

By the Cauchy-Schwarz inequality,
\begin{align*}
    \mathbbm{E}\left[\frac{1}{(1 + U)(1 + U^{'})} |\bx_{\star}, \theta, \theta^{'} \right] &\leq \sqrt{\mathbbm{E}\left[\left(\frac{1}{1 + U}\right)^2  |\bx_{\star}, \theta, \theta^{'}\right]} \sqrt{\mathbbm{E}\left[\left(\frac{1}{1 + U^{'}}\right)^2  |\bx_{\star}, \theta, \theta^{'}\right]} \\
    &\leq \left(\frac{2^{\lceil \log_2k_n \rceil + 1}}{\sqrt{n(n-1)}}\right)^2 \\
    &= \frac{4^{\lceil \log_2k_n \rceil + 1}}{n(n-1)}
\end{align*}
Following similar arguments to the calculation of the ensemble bound, the leading term can then be written as: 
\begin{align*}
&E_{\bx_{\star}, \mathcal{D}_n}[E_{\theta}\left[ H_{1}(X,\theta) \right] E_{\theta}\left[ H_{2}(X,\theta) \right] \text{Trace}[(\bx_{\mathbf{1} \mathbf{S}} - \bx_{\star \mathbf{S}} ) (\bx_{\mathbf{2} \mathbf{S}} - \bx_{\star \mathbf{S}} ) ^T]]\\
&= E[H_{1} H_{2}^{'} \text{Trace}[(\bx_{\mathbf{1} \mathbf{S}} - \bx_{\star \mathbf{S}} ) (\bx_{\mathbf{2} \mathbf{S}} - \bx_{\star \mathbf{S}} ) ^T] ] \\
&\leq \mathbbm{E}\left[ \mathbbm{E}\left[\frac{1}{(1+U)(1+U^{'})} \bigg|\bx_{\star \mathbf{S}} , \theta, \theta^{'} \right] \mathbbm{E}\left[ \mathbbm{1}_{\{\bx_{\mathbf{1} } \in A_n \}} \mathbbm{1}_{\{\bx_{\mathbf{2}} \in A_n^{'} \}} \text{Trace}[(\bx_{\mathbf{1} \mathbf{S}} - \bx_{\star \mathbf{S}} ) (\bx_{\mathbf{2} \mathbf{S}} - \bx_{\star \mathbf{S}} ) ^T] \right] \right]\\
&\leq \frac{4^{\lceil \log_2k_n \rceil + 1} }{n(n-1)} \mathbbm{E}\left[ \mathbbm{1}_{\{\bx_{\mathbf{1}} \in A_n \}} \mathbbm{1}_{\{\bx_{\mathbf{2}} \in A_n^{'} \}} \text{Trace}[(\bx_{\mathbf{1} \mathbf{S}} - \bx_{\star \mathbf{S}} ) (\bx_{\mathbf{2} \mathbf{S}} - \bx_{\star \mathbf{S}} ) ^T] \right]\\
&= \frac{4^{\lceil \log_2k_n \rceil + 1}}{n(n-1)} \mathbbm{E}\left[\mathbbm{E}\left[ \mathbbm{1}_{\{\bx_{\mathbf{1} } \in A_n \}} \mathbbm{1}_{\{\bx_{\mathbf{2}} \in A_n^{'} \}} \text{Trace}[(\bx_{\mathbf{1} \mathbf{S}} - \bx_{\star \mathbf{S}} ) (\bx_{\mathbf{2} \mathbf{S}} - \bx_{\star \mathbf{S}} ) ^T] \big|\bx_{\star \mathbf{S}}  \right] \right]\\
\end{align*}

We now look the inner expectation:
\begin{align*}
    &\mathbbm{E}\left[ \mathbbm{1}_{\{\bx_{\mathbf{1} } \in A_n \}} \mathbbm{1}_{\{\bx_{\mathbf{2}} \in A_n^{'} \}} \text{Trace}[(\bx_{\mathbf{1} \mathbf{S}} - \bx_{\star \mathbf{S}} ) (\bx_{\mathbf{2} \mathbf{S}} - \bx_{\star \mathbf{S}} ) ^T] \big|\bx_{\star \mathbf{S}}  \right] \\
    &= \mathbbm{E}\left[ \mathbbm{1}_{\{\bx_{\mathbf{1} } \in A_n \}} \mathbbm{1}_{\{\bx_{\mathbf{2}} \in A_n^{'} \}} \sum_{j}(\bx_{\mathbf{1} \mathbf{S}}^{(j)} - \bx_{\star \mathbf{S}} ^{(j)}) (\bx_{\mathbf{2} \mathbf{S}}^{(j)} - \bx_{\star \mathbf{S}} ^{(j)}) \big|\bx_{\star \mathbf{S}}  \right] \\
    &=  \sum_{j = 1}^S \mathbbm{E}\left[ \mathbbm{1}_{\{\bx_{\mathbf{1} } \in A_n \}} \mathbbm{1}_{\{\bx_{\mathbf{2}} \in A_n^{'} \}} (\bx_{\mathbf{1} \mathbf{S}}^{(j)} - \bx_{\star \mathbf{S}} ^{(j)}) (\bx_{\mathbf{2} \mathbf{S}}^{(j)} - \bx_{\star \mathbf{S}} ^{(j)})  \big|\bx_{\star \mathbf{S}} \right] \\
    &= \sum_j \mathbbm{E}\left[ \mathbbm{1}_{\{\bx_{\mathbf{1}} \in A_n \}} (\bx_{\mathbf{1} \mathbf{S}}^{(j)} - \bx_{\star \mathbf{S}} ^{(j)}) \big|\bx_{\star \mathbf{S}} \right] \times \mathbbm{E}\left[ \mathbbm{1}_{\{\bx_{\mathbf{2} } \in A_n \}} (\bx_{\mathbf{2} \mathbf{S}}^{(j)} - \bx_{\star \mathbf{S}} ^{(j)})\big|\bx_{\star \mathbf{S}}  \right] \\
    &= \sum_j \left(\mathbbm{E}\left[ \mathbbm{1}_{\{\bx_{\mathbf{1}} \in A_n \}} \left|\bx_{\mathbf{1} \mathbf{S}}^{(j)} - \bx_{\star \mathbf{S}} ^{(j)}\right|\big|\bx_{\star \mathbf{S}} \right] \right)^2
\end{align*}
where 
\begin{align*}
    \mathbbm{E}\left[ \mathbbm{1}_{\{\bx_{\mathbf{1} } \in A_n \}} \left|\bx_{\mathbf{1} \mathbf{S}}^{(j)} - \bx_{\star \mathbf{S}} ^{(j)}\right| \bigg|\bx_{\star \mathbf{S}}  \right] &= \mathbbm{E}\left[  \left|\bx_{\mathbf{1} \mathbf{S}}^{(j)} - \bx_{\star \mathbf{S}} ^{(j)}\right| \bigg| \mathbbm{1}_{\{\bx_{\mathbf{1} \mathbf{S}}^{(j)} \in A_{nj} \}}, X \right] \mathbbm{E}\left[ \mathbbm{1}_{\{\bx_{\mathbf{1} \mathbf{S}} \in A_n \}} |\bx_{\star \mathbf{S}}  \right] \\
    &= \frac{1}{2} \lambda(A_{nj}) \times 2\lambda(A_n) \\
    &= \lambda(A_{nj}) \times \lambda(A_n)
\end{align*}
We calculate the quantity $\mathbbm{E}\left[ \mathbbm{1}_{\{\bx_{\mathbf{1} } \in A_n \}} |\bx_{\star \mathbf{S}}  \right]$ by noting that 
\begin{align*}
    \mathbbm{E}\left[ \mathbbm{1}_{\{\bx_{\mathbf{1} } \in A_n \}} |\bx_{\star \mathbf{S}}  \right] &= \mathbbm{E}\left[ \mathbbm{1}_{\{\bx_{\mathbf{1} } \in A_n \}} |\bx_{\star \mathbf{S}} \in \left[0, \frac{1}{2}\right]^S \right] + \mathbbm{E}\left[ \mathbbm{1}_{\{\bx_{\mathbf{1}} \in A_n \}} |\bx_{\star \mathbf{S}} \in \left[ 1,\frac{3}{2}\right]^S \right] \\
    &= \frac{\lambda(A_n)}{\frac{1}{2}}
\end{align*}
Therefore, 
\begin{align*}
    \sum_j \left(\mathbbm{E}\left[ \mathbbm{1}_{\{\bx_{\mathbf{1} } \in A_n \}} \left|\bx_{\mathbf{1} \mathbf{S}}^{(j)} - \bx_{\star \mathbf{S}} ^{(j)}\right| \bigg|\bx_{\star \mathbf{S}} \right] \right)^2 &= \sum_j \left[\lambda(A_{nj}) \times \lambda(A_n)\right]^2\\
    &= \lambda^2(A_n) \sum_j \lambda^2(A_{nj}) \\
    &= 4^{-\lceil \log_2k_n \rceil - 2} \sum_j \left[2^{-K_{nj}}\right]^2 \\
\end{align*}

Finally, 
\begin{align*}
    &\frac{4^{\lceil \log_2k_n \rceil + 1}}{n(n-1)} \mathbbm{E}\left[ \sum_j \left(\mathbbm{E}\left[ \mathbbm{1}_{\{\bx_{\mathbf{1}} \in A_n \}} \left|\bx_{\mathbf{1} \mathbf{S}}^{(j)} - \bx_{\star \mathbf{S}} ^{(j)}\right| \bigg|\bx_{\star \mathbf{S}}  \right] \right)^2 \right]   \\
    &\leq \frac{4^{\lceil \log_2k_n \rceil + 1}}{n(n-1)} 4^{-\lceil \log_2k_n \rceil - 2} \text{S} k_n^{log_2(1 - 3p_n/4)} \\
    &= \frac{\text{S} k_n^{log_2(1 - 3p_n/4)}}{4n(n-1)}
\end{align*}

\end{enumerate}
\end{proof}

\subsection{Training sets with multiple uniformly-distributed clusters}

We finally generalize the results from Theorem 3 to training sets with $K \geq 2$ such that $K = 2^r$ for some $r \geq 1$. We continue to specify that all clusters be uniformly distributed with equal width such that all clusters represent a mean shift from one another. That is, $\X_t \stackrel{\rm i.i.d.}{\sim} \rm Unif(a_{t}, b_{t})$ with $b_t - a_t = \omega$ for all $t = 1, \ldots K$ and some constant $\omega > 0$. The test point can now be written as  $\bx_{\star \mathbf{S}}  \sim \sum_{t = 1}^K \rm Unif(a_{t}, b_{t}) \mathbbm{1}_{\{B = t \}}$ where $B \sim \rm  \text{discrete unif}\{1, K\}$ is a discrete uniform variable encoding the test point's cluster membership. 
\par We show in the following theorem that given these conditions, the upper bound for the \textit{Merged} is always larger than that of the \textit{Ensemble}, regardless of the width of each distribution $\omega$ or the number of clusters within the training set. For simplicity and comparability, we again assume that each learner predicts only on test points arising from the same distribution as in training, even though clusters may have overlapping ranges. 

\begin{theorem}\label{theorem: gen_moreclust}
Let $K = 2^{r}$ for some $r > 1$; assume that $\X_t \stackrel{\rm i.i.d.}{\sim} \rm Unif(a_{t}, b_{t})$ with $b_t - a_t = \omega$ for all $t = 1, \ldots K$. The \textit{Ensemble} is comprised of learners in which each SCL predicts only on test points arising from the same distribution as its training cluster. Then, 
\begin{enumerate}
    \item [(i)] The upper bound for the squared bias of the \textit{Ensemble} is
    \begin{align*}
    \frac{K}{4} \omega^2 S k_n^{log_2(1 - 3p_n/4)}
\end{align*}
    \item [(ii)] The upper bound for the squared bias of the \textit{Merged} is
    \begin{align*}
        4^{\log_2{k} - 1} \omega^2 S k_n^{log_2(1 - 3p_n/4)}
    \end{align*}
\end{enumerate}
\end{theorem}

\par It can be easily verified that Theorem 3 arises as special cases of the above result with $K = 2$ and $\omega = \frac{1}{2}$. We observe that as the number of clusters $K$ increases, so does the separation between the \textit{Ensemble} and \textit{Merged} bounds. In fact, the \textit{Merged} bound is $\frac{4^{\log_2{K}}}{K}$ times that of the \textit{Ensemble} bound, indicating an exponential relationship between the number of clusters and the level of improvement of the \textit{Ensemble} over the \textit{Merged}. Interestingly, the \textit{Merged} has the same bound as would be achieved by a single forest trained on observations arising from a $\rm Unif(a, a + K\omega)$ distribution for some constant $a$; that is, the \textit{Merged} ignores the cluster-structure of the data and instead treats the data as arising from the average of its' component distributions.

For the proof, we follow the same framework as in the above two sections in order to first present upper bounds for the \textit{Ensemble} and the \textit{Merged} learners for data with two clusters generalized to any distributional width $\omega > 0$, and then show that this approach can be extended to any $K$ that can be expressed as a power of 2.  For $K = 2$, we remark on the generalized representation regarding the measure of the leaf node containing the test point. 
\begin{remark}
\label{rem: gen_leafsize}
For $K = 2$, let $F_1 = \rm Unif(a_{1}, b_{1})$ and $F_2 = \rm Unif(a_{2}, b_{2})$. The test point $\bx_{\star \mathbf{S}}  \sim F_{\bx_{\star \mathbf{S}} }$, where $F_{\bx_{\star \mathbf{S}} } = AF_1 + (1-A)F_2$ and $A \sim $ Bernoulli$(\frac{1}{2})$. $F_{\bx_{\star \mathbf{S}} ^{(j)}}$ similarly denotes the marginal distribution of $\bx_{\star \mathbf{S}} ^{(j)}$. To calculate expressions for the endpoints of the box containing $\bx_{\star \mathbf{S}} $, first delineate
\begin{align*}
    t_{nj} &\overset{D}{=} \sum_{k = 1}^{K_{nj}} B_{kj} 2^{-k}   \\
    s_{nj} &\overset{D}{=} \sum_{k = 1}^{K_{nj}} B_{kj} 2^{-k} + 2^{-{ K_{nj}}} 
\end{align*}
where $B_{kj}$ and $K_{nj}$ follow the same definitions as before. 
Then, since $t_{nj}$ and $s_{nj}$ converge to standard uniform variables as $K_{nj} \to \infty$, we can apply the inverse probability transform to express the endpoints $a_{nj}$ and $b_{nj}$ as
\begin{align*}
    a_{nj} &= F_{\bx_{\star \mathbf{S}} }^{-1}(t_{nj}) = A F_{\bx_1^{(j)}}^{-1}(t_{nj}) + (1 - A) F_{\bx_2^{(j)}}^{-1}(t_{nj}) \\
    b_{nj} &= F_{\bx_{\star \mathbf{S}} }^{-1}(s_{nj}) = A F_{\bx_1^{(j)}}^{-1}(s_{nj}) + (1 - A) F_{\bx_2^{(j)}}^{-1}(s_{nj})
\end{align*}
Note that for $\bx \sim \rm Unif(a, b)$ distribution with $b - a = \omega$, $F_{\bx}^{-1}(u) = \omega u + a$. 
Therefore, the length of each side of the box is now given by 
\begin{align*}
    \lambda(A_{nj}) &= b_{nj} - a_{nj} \\
    &= A\left[F_{\bx_1^{(j)}}^{-1}(s_{nj}) -   F_{\bx_1^{(j)}}^{-1}(t_{nj}) \right] + (1-A)\left[F_{\bx_2^{(j)}}^{-1}(s_{nj}) - F_{\bx_2^{(j)}}^{-1}(t_{nj}) \right] \\
    &= \omega 2^{-K_{nj}}
\end{align*}
\end{remark}

We can then proceed with the proof of Theorem \ref{theorem: gen_moreclust}. 
\begin{proof}
\begin{enumerate}
    \item [(i)] 
    For the \textit{Ensemble}, we begin by considering training and testing data from the first cluster.
\begin{align*}
    E\left[ \mathbbm{1}_{\{\bx_{\mathbf{i}} \in A_n \}} \mathbbm{1}_{\{i \in \mathbb{S}_1 \}} \mathbbm{1}_{\{A= 1 \}} \right] &=  E\left[ \mathbbm{1}_{\{\bx_{\mathbf{i}} \in A_n \}} | \mathbbm{1}_{\{i \in \mathbb{S}_1 \}} \mathbbm{1}_{\{A= 1 \}} \right] P\left( i \in \mathbb{S}_1 \right) P\left( A= 1 \right) \\
    &= \frac{1}{4}  \prod_{j = 1}^S E\left[ \mathbbm{1}_{\{\bx_{\mathbf{i}}^{(j)} \in A_{nj} \}} | \mathbbm{1}_{\{i \in \mathbb{S}_1 \}} \mathbbm{1}_{\{A= 1 \}} \right] \\
    &= \frac{1}{4}  \prod_{j = 1}^S \left[ P\left( \bx_{\mathbf{i}}^{(j)} \leq b_{nj} | i \in \mathbb{S}_1, A= 1\right) - P\left(\bx_{\mathbf{i}}^{(j)} \leq a_{nj} | i \in \mathbb{S}_1, A= 1 \right) \right] \\
    &= \frac{1}{4}  \prod_{j = 1}^S F_{X_{\mathbf{1}}^{(j)}}\left(F_{\bx_{\mathbf{1} \mathbf{S}}^{(j)}}^{-1}\left(s_{nj} \right) \right) - F_{X_{\mathbf{1}}^{(j)}}\left(F_{\bx_{\mathbf{1} \mathbf{S}}^{(j)}}^{-1}\left(t_{nj} \right) \right) \\
    &=  \frac{1}{4}  \prod_{j = 1}^S [s_{nj} - t_{nj}] \\
    &= 2^{-\lceil \log_2k_n \rceil-2}
\end{align*}
The quantity $E\left[ \mathbbm{1}_{\{\bx_{\mathbf{i}} \in A_n \}} \mathbbm{1}_{\{i \in \mathbb{S}_2 \}} \mathbbm{1}_{\{A = 0 \}} \right]$ calculated with training and testing points from the second cluster, takes on the same value. 
Therefore, the Cauchy Schwarz bound corresponding to either of these situations is equal to
\begin{align*}
    \frac{4^{\lceil \log_2k_n \rceil+2}}{n(n-1)}
\end{align*}
Then, using similar arguments as before, 
\begin{align*}
    &\mathbbm{E}_{\bx_{\star \mathbf{S}} , \mathcal{D}_n, \theta, \theta^{'}}\left[ W_{1} W_{2}^{'} \text{Trace}[(\bx_{\mathbf{1} \mathbf{S}} - \bx_{\star \mathbf{S}} ) (\bx_{\mathbf{2}} - \bx_{\star \mathbf{S}} )^T] \right] \\
    &= \frac{4^{\lceil \log_2k_n \rceil+2}}{n(n-1)} \\
    &\times \Bigg[ \mathbbm{E}_{\bx_{\mathbf{1}}, \bx_{\mathbf{2}}}\bigg[\mathbbm{1}_{\{\bx_{\mathbf{1} \mathbf{S}} \in A_n\}} \mathbbm{1}_{\{\bx_{\mathbf{2}} \in A_n^{'}\}} \left[\mathbbm{1}_{\{1 \in \mathbb{S}_1\}}\right] \left[\mathbbm{1}_{\{2 \in \mathbb{S}_1^{'}\}}\right] \left[\mathbbm{1}_{\{A= 1\}}\right]^2 \text{Trace}[(\bx_{\mathbf{1} \mathbf{S}} - \bx_{\star \mathbf{S}} ) (\bx_{\mathbf{2}} - \bx_{\star \mathbf{S}} )^T] \bigg] \\
    &+ \mathbbm{E}_{\bx_{\mathbf{1}}, \bx_{\mathbf{2}}}\bigg[\mathbbm{1}_{\{\bx_{\mathbf{1} \mathbf{S}} \in A_n\}} \mathbbm{1}_{\{\bx_{\mathbf{2}} \in A_n^{'}\}} \left[\mathbbm{1}_{\{1 \in \mathbb{S}_2\}}\right] \left[\mathbbm{1}_{\{2 \in \mathbb{S}_2^{'}\}}\right] \left[\mathbbm{1}_{\{A= 1\}}\right]^2 \text{Trace}[(\bx_{\mathbf{1} \mathbf{S}} - \bx_{\star \mathbf{S}} ) (\bx_{\mathbf{2}} - \bx_{\star \mathbf{S}} )^T] \bigg] 
    \Bigg]
\end{align*}
We next need to calculate
\begin{align*}
    &E\left[ \left|\bx_{\mathbf{1} \mathbf{S}}^{(j)} - \bx_{\star \mathbf{S}} ^{(j)}|\mathbbm{1}_{\{\bx_{\mathbf{i}} \in A_n \}} \mathbbm{1}_{\{i \in \mathbb{S}_1 \}} \mathbbm{1}_{\{A= 1 \}} \right|\bx_{\star \mathbf{S}}  \right] \\
    &= E\left[ \left|\bx_{\mathbf{1} \mathbf{S}}^{(j)} - \bx_{\star \mathbf{S}} ^{(j)}\right| \mathbbm{1}_{\{\bx_{\mathbf{i}} \in A_n \}} \mathbbm{1}_{\{i \in \mathbb{S}_1 \}} \mathbbm{1}_{\{A= 1 \}}\right] E\left[\mathbbm{1}_{\{\bx_{\mathbf{i}} \in A_n \}} | \mathbbm{1}_{\{i \in \mathbb{S}_1 \}} \mathbbm{1}_{\{A= 1 \}}\right]  P\left( i \in \mathbb{S}_1 \right) P\left( A= 1 \right) \\
    &= \frac{1}{2} \lambda(A_{n} | A= 1) \times 2^{-\lceil \log_2k_n \rceil-2} \\
    &= 2^{-\lceil \log_2k_n \rceil-3} \omega 2^{-K_{nj}}
\end{align*}
We can then simplify the following expressions: 
\begin{align*}
    &E\left[ \sum_j \left(\mathbbm{E}\left[\left|\bx_{\mathbf{1} \mathbf{S}}^{(j)} - \bx_{\star \mathbf{S}} ^{(j)}\right|\mathbbm{1}_{\{\bx_{\mathbf{i}} \in A_n \}} \mathbbm{1}_{\{i \in \mathbb{S}_1 \}} \mathbbm{1}_{\{A= 1 \}} \big| \bx_{\star \mathbf{S}} \right] \right)^2 \right] \\
    &= 4^{-\lceil \log_2k_n \rceil-3} \omega^2 \sum_j E\left[ \left( 2^{-K_{nj}} \right)^2 \right]
\end{align*}
and 
\begin{align*}
   &E\left[ \sum_j \left(\mathbbm{E}\left[\left|\bx_{\mathbf{1} \mathbf{S}}^{(j)} - \bx_{\star \mathbf{S}} ^{(j)}\right|\mathbbm{1}_{\{\bx_{\mathbf{i}} \in A_n \}} \mathbbm{1}_{\{i \in \mathbb{S}_2 \}} \mathbbm{1}_{\{A = 0 \}} \big| \bx_{\star \mathbf{S}} \right] \right)^2 \right] \\ &= 4^{-\lceil \log_2k_n \rceil-3} \omega^2 \sum_j E\left[ \left( 2^{-K_{nj}} \right)^2 \right]
\end{align*}

Therefore, the overall upper bound for the ensemble for $K = 2$ is 
\begin{align*}
 &n(n-1) \times \frac{4^{\lceil \log_2k_n \rceil+2}}{n(n-1)} \times 4^{-\lceil \log_2k_n \rceil-3} \times \omega^2 \sum_{j = 1}^S 2 E\left[ \left( 2^{-K_{nj}} \right)^2 \right]\\
 &= \frac{1}{2} \omega^2 S  \sum_{j = 1}^S k_n^{log_2(1 - 3p_n/4)} 
\end{align*}
We can apply the same arguments to the $K =4$ situation, yielding 

\begin{align*}
    &\frac{1}{4} \omega^2 \sum_{j =1}^S 4  E\left[ \left( 2^{-K_{nj}} \right)^2 \right]\\
    &=  \omega^2 S k_n^{log_2(1 - 3p_n/4)} \\
\end{align*}

In a similar way, bounds may be calculated for $K = 8, 16, ...$; we can thus generalize the above bound to $K \geq 2$, assuming $K$ is a power of 2.

    \item [(ii)] 
    For the \textit{Merged}, we again begin with $K = 2$. The training data now takes the form
\begin{align*}
    \bx_{\mathbf{i}} \sim F_1 \mathbbm{1}_{\{ i \in \mathbb{S}_1 \}} + F_2 \mathbbm{1}_{\{ i \in \mathbb{S}_2 \}}
\end{align*}
We first calculate
\begin{align*}
    E\left[\mathbbm{1}_{\{\bx_{\mathbf{i}}^{(j)} \in A_{nj} \}} \right] &= P(A= 1) P(i \in \mathbb{S}_1) E\left[\mathbbm{1}_{\{\bx_{\mathbf{i}}^{(j)} \in A_{nj} \}} |A= 1, i \in \mathbb{S}_1 \right] \\
    &+ P(A = 0)P(i \in \mathbb{S}_1) E\left[\mathbbm{1}_{\{\bx_{\mathbf{i}}^{(j)} \in A_{nj} \}} |A = 0, i \in \mathbb{S}_1 \right] \\
    &+ P(A= 1) P(i \in \mathbb{S}_1) E\left[\mathbbm{1}_{\{\bx_{\mathbf{i}}^{(j)} \in A_{nj} \}} |A= 1, i \in \mathbb{S}_1 \right] \\
    &+ P(A = 0)P(i \in \mathbb{S}_2) E\left[\mathbbm{1}_{\{\bx_{\mathbf{i}}^{(j)} \in A_{nj} \}} |A = 0, i \in \mathbb{S}_2 \right] \\
    &\leq \frac{1}{4}\left[4(s_{nj} - t_{nj}) \right] \\
    &= 2^{-K_{nj}}
\end{align*}

Therefore, $E\left[\mathbbm{1}_{\{X_{i} \in A_{n} \}} \right] \leq 2^{-\lceil \log_2k_n \rceil}$. The Cauchy-Schwarz bound is then equal to 
\begin{align*}
    \frac{4^{\lceil \log_2k_n \rceil}}{n(n-1)}
\end{align*}
The next quantity we calculate is
\begin{align*}
    \mathbbm{E}\left[ \mathbbm{1}_{\{\bx_{\mathbf{1} \mathbf{S}} \in A_n \}} \left|\bx_{\mathbf{1} \mathbf{S}}^{(j)} - \bx_{\star \mathbf{S}} ^{(j)}\right| \bigg| \bx_{\star \mathbf{S}}  \right] &= \mathbbm{E}\left[  \left|\bx_{\mathbf{1} \mathbf{S}}^{(j)} - \bx_{\star \mathbf{S}} ^{(j)}\right| \bigg| \mathbbm{1}_{\{\bx_{\mathbf{1} \mathbf{S}}^{(j)} \in A_{nj} \}}, \bx_{\star \mathbf{S}}  \right] \mathbbm{E}\left[ \mathbbm{1}_{\{\bx_{\mathbf{1} \mathbf{S}} \in A_n \}} |\bx_{\star \mathbf{S}}  \right] \\
    &= \frac{1}{2} \lambda(A_{nj}) \times 2^{- \lceil \log_2k_n\rceil + 1} \\
\end{align*}
Therefore, 
\begin{align*}
    \sum_j \left(\mathbbm{E}\left[ \mathbbm{1}_{\{\bx_{\mathbf{1} \mathbf{S}} \in A_n \}} \left|\bx_{\mathbf{1} \mathbf{S}}^{(j)} - \bx_{\star \mathbf{S}} ^{(j)}\right| \bigg| \bx_{\star \mathbf{S}} \right] \right)^2 &= \sum_j \left[\lambda(A_{nj}) \times 2^{- \lceil \log_2k_n\rceil}\right]^2\\
    &= 4^{- \lceil \log_2k_n\rceil} \sum_j \lambda^2(A_{nj}) \\
\end{align*}
To move forward, we first need to calculate 
\begin{align*}
    \E\left[\lambda^2(A_{nj}) \right] &=   \E\left[\left(A\left[F_{\bx_{\mathbf{1}}^{(j)}}^{-1}(s_{nj}) -   F_{\bx_{\mathbf{1}}^{(j)}}^{-1}(t_{nj}) \right] + (1-A)\left[F_{\bx_{\mathbf{2}}^{(j)}}^{-1}(s_{nj}) - F_{\bx_{\mathbf{2}}^{(j)}}^{-1}(t_{nj}) \right]\right)^2 \right] \\
    &=\E\left[A^2 \left( F_{\bx_{\mathbf{1}}^{(j)}}^{-1}(s_{nj}) -   F_{\bx_{\mathbf{1}}^{(j)}}^{-1}(t_{nj}) \right)^2 + (1-A)^2 \left( F_{\bx_{\mathbf{2}}^{(j)}}^{-1}(s_{nj}) - F_{\bx_{\mathbf{2}}^{(j)}}^{-1}(t_{nj}) \right)^2 \right]  \\
    &= \frac{1}{2} E\left[ 2 \left( \omega 2^{-K_{nj}} \right)^2  \right]
\end{align*}
The cross-terms in the quadratic cancel out since $(1-A)
\times A = 0$ no matter what the value of $A$. Then, 
\begin{align*}
    &\mathbbm{E}\left[ \sum_j \left(\mathbbm{E}\left[ \mathbbm{1}_{\{\bx_{\mathbf{1} \mathbf{S}} \in A_n \}} \left|\bx_{\mathbf{1} \mathbf{S}}^{(j)} - \bx_{\star \mathbf{S}} ^{(j)}\right| \bigg| \bx_{\star \mathbf{S}}  \right] \right)^2 \right]  \\
    &= 4^{- \lceil \log_2k_n\rceil} \times \sum_j \E\left[ \left( \omega 2^{-K_{nj}} \right)^2 \right] \\
     &= 4^{- \lceil \log_2k_n\rceil}  \omega^2 S k_n^{log_2(1 - 3p_n/4)} 
\end{align*}
Finally, the overall bound for $K = 2$ is 
\begin{align*}
    &n(n-1) \times \frac{4^{\lceil \log_2k_n \rceil}}{n(n-1)} \times 4^{- \lceil \log_2k_n\rceil} \times \sum_{j =1}^S  E\left[ \left(\omega 2^{-K_{nj}} \right)^2 \right] \\
    &= \omega^2 S k_n^{log_2(1 - 3p_n/4)} 
\end{align*}

We can analogously calculate the bound for $K = 4$:
\begin{align*}
     &4 \omega^2 \sum_j E\left[ \left( 2^{-K_{nj}} \right)^2 \right]\\
     &= 4 \omega^2 S k_n^{log_2(1 - 3p_n/4)} 
\end{align*}
\end{enumerate}
For the \textit{Merged}, note that 
\begin{align*}
    \mathbbm{E}\left[ \mathbbm{1}_{\{\bx_{\mathbf{1} \mathbf{S}} \in A_n \}} (\bx_{\mathbf{1} \mathbf{S}}^{(j)} - \bx_{\star \mathbf{S}} ^{(j)}) \bigg| \bx_{\star \mathbf{S}}  \right] &= \mathbbm{E}\left[  (\bx_{\mathbf{1} \mathbf{S}}^{(j)} - \bx_{\star \mathbf{S}} ^{(j)}) \bigg| \mathbbm{1}_{\{\bx_{\mathbf{1} \mathbf{S}}^{(j)} \in A_{nj} \}}, \bx_{\star \mathbf{S}}  \right] \mathbbm{E}\left[ \mathbbm{1}_{\{\bx_{\mathbf{1} \mathbf{S}} \in A_n \}} |\bx_{\star \mathbf{S}}  \right] \\
    &\leq \frac{1}{2} \lambda(A_{nj}) \times 2^{- \lceil \log_2k_n\rceil+ 1} \\
\end{align*}
and secondly, in calculating $\E\left[\lambda^2(A_{nj}) \right]$, the cross terms always equal zero, since $\mathbbm{1}_{\{B = l\}}\mathbbm{1}_{\{B = k\}}$ always equals zero for $l \neq k$. We generalize the above results to higher values of $K$ by noting that the $\frac{1}{4}$ multiplying factor in (i) always remains constant, while the multiplier of  $E\left[ \left(\omega 2^{-K_{nj}} \right)^2 \right]$ is equal to $K$. The \textit{Merged} bound is generalized by noting as in the multiplying term increases exponentially with $K$, at a rate of $4^{\log_2{k} - 1}$. Formalizing these two facts completes the proof. 
\end{proof}

\end{document}